\documentclass{article}

\usepackage[latin1]{inputenc}
\usepackage{subfigure}
\usepackage{amsmath, amssymb, amsthm}
\usepackage{enumitem}
\usepackage{hyperref}
\usepackage{graphicx}
\usepackage{fullpage}
\usepackage{array}

\usepackage[noabbrev]{cleveref}

\newtheorem{thm}{Theorem}[section]
\newtheorem{prop}[thm]{Proposition}
\newtheorem{cor}[thm]{Corollary}
\newtheorem{exa}{Example}[section]
\newtheorem{lemma}[thm]{Lemma}
\newtheorem{rem}{Remark}[section] 
\newcommand{\rationals}{\mathbb{Q}}

\author{Guillaume Chapuy
and  Wenjie Fang
\thanks{ 
Both authors acknowledge partial  support from \emph{Agence Nationale de la Recherche}, grant number ANR~12-JS02-001-01 ``Cartaplus'', and from the City of Paris, grant ``\'Emergences Paris 2013, Combinatoire \`a Paris''.
G.C. is currently visiting McGill University, School of Computer Science, Montr\'eal, Canada, and acknowledges partial support from Luc Devroye.
W.F. acknowledges hospitality from LaBRI, UMR CNRS 5800, Universit\'e Bordeaux-I, France.
Emails:~{\tt \{Guillaume.Chapuy,Wenjie.Fang\}@liafa.univ-paris-diderot.fr}.
}\\
{\small LIAFA, UMR CNRS 7089, 
  Universit\'e Paris-Diderot, Paris Cedex 13, France.}
}
\title{Generating functions of bipartite maps on orientable surfaces}
\begin{document}
\maketitle
\begin{abstract}
We compute, for each genus $g\geq 0$, the generating function $L_g\equiv L_g(t;p_1,p_2,\dots)$ of (labelled) bipartite maps on the orientable surface of genus $g$, with control on all face degrees. We exhibit an explicit change of variables such that for each $g$, $L_g$ is a rational function in the new variables, computable by an explicit recursion on the genus. The same holds for the generating function $F_g$ of \emph{rooted} bipartite maps. The form of the result is strikingly similar to the Goulden/Jackson/Vakil and Goulden/Guay-Paquet/Novak formulas for the generating functions of classical and monotone Hurwitz numbers respectively, which suggests stronger links between these models. Our result  complements recent results of Kazarian and Zograf, who studied the case where the number of faces is bounded, in the equivalent formalism of \textit{dessins d'enfants}. Our proofs borrow some ideas from Eynard's ``topological recursion'' that he applied in particular to even-faced maps (unconventionally called ``bipartite maps'' in his work). However, the present paper requires no previous knowledge of this topic and comes with elementary (complex-analysis-free) proofs written in the perspective of formal power series.
\end{abstract}


\section{Introduction}
\label{sec:in}

A \emph{map of genus $g\geq0$} is a graph embedded into the $g$-torus (the sphere with $g$ handles attached), in such a way that the connected components the complement of the graph are simply connected. See~\Cref{sec:defs} for complete definitions. The enumeration of maps is a classical topic in combinatorics, motivated both from the beautiful enumerative questions it unveils, and by its many connections with other areas of mathematics, see \textit{e.g.}~\cite{LandoZvonkine}.
The enumeration of \emph{planar} maps (when the underlying surface is the sphere) was initiated by Tutte who showed~\cite{Tutte:census} that the generating function $Q_0(t)$ of rooted planar maps by the number of edges is an algebraic function given by:
\begin{align}\label{eq:introTutte}
Q_0(t) = s(4-s)/3 \mbox{ where } s=1+3ts^2.
\end{align}
The enumeration of planar maps has since grown into an enormous field of research on its own, out of the scope of this introduction, and we refer to~\cite{Schaeffer:survey} for an introduction and references.

The enumeration of maps on surfaces different from the sphere was pioneered by Bender and Canfield, who showed~\cite{BC} that for each $g\geq 1$, the generating function $Q_g(t)$ of rooted maps embedded on the $g$-torus (see again \Cref{sec:defs} for definitions) is a \emph{rational} function of the parameter $s$ defined in~\eqref{eq:introTutte}. For example, for the torus, one has
$ 
Q_1(t) = 
\frac{1}{3}{\frac {s \left( s-1 \right) ^{2}}{ \left( s+2 \right)  \left( s-
2 \right) ^{2}}}. $
This deep and important result was the first of a series of rationality results established for generating functions of maps or related combinatorial objects on higher genus surfaces. Gao~\cite{Gao} proved several rationality results for the generating functions of maps with prescribed degrees using a variant of the kernel method (see~\Cref{rem:oldKernel} for a comment about this). Later, Goulden, 
Jackson and Vakil~\cite{GJV} proved a rationality statement for the generating functions of Hurwitz numbers (an algebraic model having many connections with map enumeration) relying on deep algebraic results~\cite{ELSV}. 
More recently, Goulden, Guay-Paquet, and Novak~\cite{GGPN} introduced a variant called \emph{monotone} Hurwitz numbers, for which they proved a rationality statement very similar to the one of~\cite{GJV}.
We invite the reader to compare our main result (\Cref{thm:mainUnrooted}) with \cite[Thm. 3.2]{GJV} and \cite[Thm. 1.4]{GGPN} (see also~\cite[Sec.~1.5]{GGPN}). The analogy between those results is striking and worth further investigation.

In parallel to this story, mathematical physicists have developed considerable tools to attack problems of map enumeration, motivated by their many connections with high energy physics, and notably matrix integrals (see e.g. \cite[Chapter 5]{LandoZvonkine}).
Among them, the \emph{topological recursion} is a general framework developed by Eynard and his school~\cite{Eynard:book, EO}, that gives, in a universal way, the solution to many models related to map enumeration and algebraic geometry, see \cite{EO}. In his book~\cite[Chap. 3]{Eynard:book}, Eynard applies this technique to the enumeration of maps on surfaces, and obtains in particular a rationality theorem for generating functions of \emph{even maps}, i.e., maps with faces of even degrees (that he, unconventionally, calls ``bipartite'' maps, although the two models are different). The proofs in these references use a complex-analytic viewpoint, and are often not easy to read for the pure combinatorialist, especially given the fact that they are published in the mathematical physics literature. 

The main purpose of the present paper is to establish a rationality theorem for \emph{bipartite} maps, which is a very natural and widely considered model of maps from both the topological and combinatorial viewpoint, see \Cref{sec:defs}.  Our proof recycles two ideas of the topological recursion,
however previous knowledge of the latter is not required, and our proofs rely only on a concrete viewpoint on Tutte equations and on \emph{formal} power series. We hope to make some of the key ideas of the topological recursion more accessible to pure combinatorialists, using a language that enables an easier comparison with the traditional combinatorial approaches. To be precise, the two crucial steps that are directly inspired from the topological recursion, and that differ from traditional kernel-like methods often used by combinatorialists are \Cref{prop:noPoles!} and \Cref{thm:toprec}. Once these two results are proved (with a formal series viewpoint), an important part of the work deals then with making explicit the auxiliary variables that underlie the rationality statements (the ``Greek'' variables in \Cref{thm:mainUnrooted} below). This is done in Sections~\ref{sec:Gamma} and~\ref{sec:Y}. Finally, the proof of the ``integration step'' needed to prove our statement in the labelled case (\Cref{thm:mainUnrooted}) from the rooted case (\Cref{thm:mainRooted}) is an \textit{ad hoc} proof, partly relying on a bijective insight from~\cite{Chapuy:constellations}, see \Cref{sec:unrooting}. This approach has the advantage of giving a partial combinatorial interpretation to the absence of logarithm in generating functions of unrooted maps in genus higher than 1 (Theorem~\ref{thm:mainUnrooted}).

Bipartite maps have been considered before in the literature. The first author studied them by \emph{bijective} methods~\cite{Chapuy:constellations}, and obtained rationality statements that are weaker than the ones we obtain with generating functions here. More recently, Kazarian and Zograf~\cite{KZ}, using a variant of the topological recursion, proved a \emph{polynomiality} statement for the generating functions of bipartite maps with \emph{finitely many} faces (these authors deal with \emph{dessins d'enfants} rather than bipartite maps, but the two models are equivalent, see \cite[Chap. 1]{LandoZvonkine}). On the contrary our main result covers the case of arbitrarily many faces, which is  more general. Indeed, not only does it prove that each fixed-face generating functions is a polynomial in our chosen set of parameters (by a simple derivation), but it also gives a very strong information on the mutual dependency of these different generating functions. Note however that \cite{KZ} keeps track of one more variable (keeping control on the number of vertices of each color in their expressions). It is probably possible to unify the two results together.
 
To finish this introduction,  and to prevent a misunderstanding, 
we mention that the generating functions of bipartite maps of all genera can be collected into a grand generating function that is known to be a Tau-function of the KP (and even 2-Toda) hierarchy, see e.g. \cite{GJ:KP}. This fact does not play any role in the present paper, and we do not know how to use it to study the kind of questions we are interested in here. However, if this was possible, this could lead to recurrence formulas to compute the generating functions that would be more efficient than the ones we obtain here, as was done so far only in the two very special cases of triangulations~\cite{GJ:KP} and bipartite quadrangulations~\cite{CC}.

This paper is organized as follows. In \Cref{sec:defsMain}, we give necessary definitions and notation, and we state the main results (Theorems~\ref{thm:mainUnrooted} and~\ref{thm:mainRooted}). In \Cref{sec:Tutte} we write the Tutte/loop equation, and we admit a list of propositions and lemmas, without proof, that enable us to prove \Cref{thm:mainRooted}. The proofs of these admitted propositions and lemmas are fully given in \Cref{sec:Gamma} and \Cref{sec:Y}. Finally, \Cref{sec:unrooting} gives the proof of \Cref{thm:mainUnrooted},
and \Cref{sec:comments} collects some final comments.

\noindent{\bf Acknowledgements:} Both authors thank Mireille Bousquet-M\'elou for interesting discussions
and comments.

\section{Surfaces, maps, and the main results}
\label{sec:defsMain}
 
\subsection{Surfaces, maps.}\label{sec:defs}
In this paper, a \emph{surface} is a connected, compact, oriented 2-manifold without boundary, considered up to oriented homeomorphism. For each integer $g\geq 0$, we let $\mathcal{S}_g$ be \emph{$g$-torus}, that is obtained from the $2$-sphere $\mathcal{S}_0$ by adding $g$ handles. Hence $\mathcal{S}_1$ is the torus, $\mathcal{S}_2$ is the double torus, etc. By the theorem of classification, each surface is homeomorphic to one of the surfaces $\mathcal{S}_g$ for some $g\geq 0$ called its \emph{genus}.

A \emph{map} is a graph $G$ (with loops and multiple edges allowed) properly embedded into a surface $\mathbb{S}$, in such a way that the connected components of $\mathbb{S} \setminus G$, called \emph{faces}, are topological disks. The \emph{genus} of a map is the genus of the underlying surface. A map is \emph{bipartite} if vertices of its underlying graph are coloured in black and white such that there is no monochromatic edge. Note that a bipartite map may have multiple edges but no loops.
A map is \emph{rooted} if an edge (called the \emph{root edge}) is distinguished and oriented. The origin of the root edge is the \emph{root vertex}, and the face incident to the right of the (oriented) root edge is the \emph{root face}. By convention the root vertex of a bipartite map is always coloured white.
We consider rooted maps up to oriented homeomorphisms preserving the root edge and its orientation.
The \emph{degree} of a vertex in a bipartite map is its degree in the underlying multigraph, i.e. the number of edges incident to it, with multiplicity. The \emph{degree} of a face in a bipartite map is the number of edges bounding this face, counted with multiplicity. Because colors alternate along an edge, the degree of faces in a bipartite map are all even numbers
\footnote{Note, however that the converse is true only for genus 0: for each genus $g\geq 1$, there exist maps with all faces of even degree but not bipartite (in genus $1$, and example is the $m\times n$ square grid with toroidal identifications, when $m$ or $n$ is odd).}.
If a bipartite map has $n$ edges, the sum of all face-degrees is equal to $2n$,
 and the sum of all vertex-degrees of each given color is equal to $n$.

From the algebraic viewpoint (and for the comparison with Hurwitz numbers as defined in~\cite{GJV,GGPN}), it is sometimes convenient to consider a variant of rooted maps called labelled maps. A \emph{labelled} bipartite map of size $n$ is a bipartite map with $n$ edges equipped with a labelling of its edges from $1$ to $n$ such that its root edge receives label~$1$. There is a $1$-to-$(n-1)!$ correspondence between rooted bipartite maps and labelled bipartite maps of size $n$. Given a labelled bipartite map, one can define two permutations $\sigma_\circ$ and $\sigma_\bullet$ in $\mathfrak{S}_n$ whose cycles record the counter-clockwise ordering of edges around white and black vertices, respectively. 
See Figure~\ref{fig:bipmap}.
This is a bijection between labelled bipartite maps of size $n$ and pairs $(\sigma_\circ, \sigma_\bullet)$ of permutations in $\mathfrak{S}_n$ such that the subgroup $\langle \sigma_\circ, \sigma_\bullet \rangle\subset\mathfrak{S}_n$ acts transitively on $[1\dots n]$. In this correspondence, cycles of $\sigma_\circ, \sigma_\bullet$, and $\sigma_\circ\sigma_\bullet$ are in natural correspondence with white vertices, black vertices, and faces, and the lengths of these cycles correspond to degrees (for vertices) and half-degrees (for faces). The genus $g$ of the underlying surface is related to the number of cycles of the three permutations $\sigma_\circ$, $\sigma_\bullet$ and $\sigma_\circ\sigma_\bullet$ by Euler's formula:
\begin{align*}
\ell(\sigma_\circ)+\ell(\sigma_\bullet)+\ell(\sigma_\circ\sigma_\bullet)
=n+2-2g.
\end{align*}
\begin{figure}\centering
\includegraphics{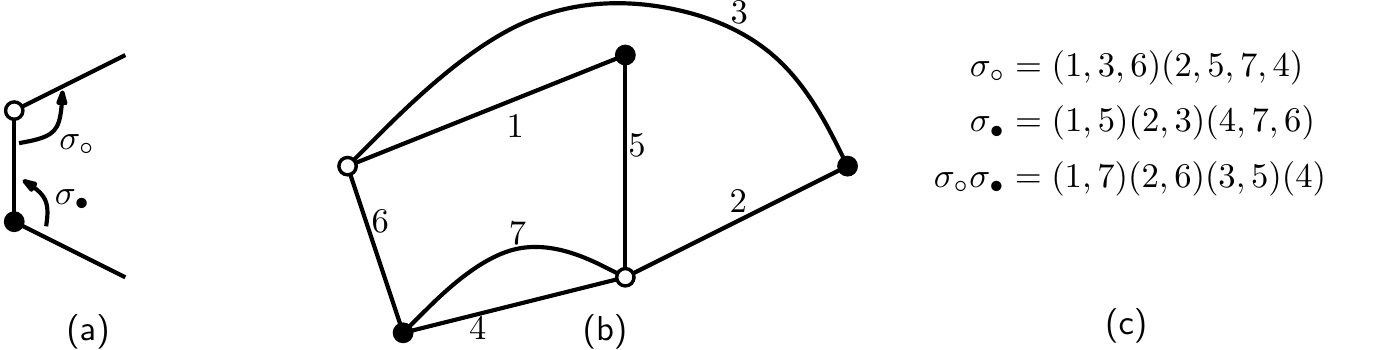}
\caption{A labelled bipartite map with 7 edges in the plane, and the corresponding permutations $\sigma_\circ, \sigma_\bullet, \sigma_\circ\sigma_\bullet$.}\label{fig:bipmap}

\end{figure}

\subsection{Notation for series and changes of variables}
\label{sec:notation}
In this paper,  $t$, $x$, and $p_1,p_2, \dots$ are indeterminates.  Indices of the variables $(p_i)_{i\geq 1}$ will be extended multiplicatively from integers to integer partitions, for example $p_{3,3,1}=p_1 (p_3)^2$, and the same convention will be used for other indexed sequences of variables in the paper, such as $(\eta_i)_{i\geq 1}$ or $(\zeta_i)_{i\geq 1}$.

If $\mathbb{B}$ is a ring (or field) and $s$ an indeterminate, we denote by $\mathbb{B}[s]$, $\mathbb{B}(s)$, $\mathbb{B}[[s]]$, $\mathbb{B}((s))$, $\mathbb{B}((s^*))$ the ring (or field) of polynomials, rational functions, formal power series (f.p.s.), formal Laurent series, and Puiseux series in $s$ with coefficients in $\mathbb{B}$, respectively. If $\mathbb{B}$ is a field, $\overline{\mathbb{B}}$ is its algebraic closure.
We will often omit the dependency of generating functions on the variables
in the notation, for example we will write $L_g$ for $L_g(t; p_1,p_2,\dots)$ and $F_g$ for $F_g(t;x;p_1,p_2,\dots)$.
In this paper all fields have characteristic~$0$.

Finally, an important role  will be played by the ``change of variables'' $(t,x)\leftrightarrow (z,u)$ given by the equations:
\begin{align}
\label{eq:z}
\displaystyle
z=t\left(1+\sum_{k\geq 1} {2k-1 \choose k} p_k z^k \right),
\end{align}
\begin{align}\label{eq:u}
u=x(1+zu)^2.
\end{align}
These equations define two unique f.p.s. $z\equiv z(t) \in \mathbb{Q}[p_1,p_2,\dots][[t]]$ and $u\equiv u(t,x)\in\mathbb{Q}[x, p_1,p_2,\dots][[t]]$. Moreover, this change of variables is reversible, via $t(z)=\frac{z}{1+\sum_k {2k-1 \choose k} p_k z^k}$ and $x(z,u)=\frac{u}{(1+zu)^2}$.  Note also that, if $H\equiv H(t,x)\in \mathbb{B}[x][[t]]$ is a f.p.s. in $t$ with polynomial coefficients in $x$ over some ring $\mathbb{B}$ containing all $p_i$, then $H(t(z),x(z,u))$ is an element of $\mathbb{B}[u][[z]]$. In this paper \emph{we are going to abuse notation} and we will switch without warning between a series $H\in \mathbb{B}[x][[t]]$ and its image in $\mathbb{B}[u][[z]]$ via the change of variables. We are going to use the single letter $H$ for both objects, relying on the context that should prevent any misunderstanding.

\subsection{Generating functions and the main result}

For $n \geq1$ and $\mu$ a partition of $n$ (denoted as $\mu \vdash n$), let $\mathfrak{l}_g(\mu)$ be the number of labelled bipartite maps of size $n$ and genus $g\geq 0$ whose half-face degrees are given by the parts of $\mu$. Equivalently:
$$
\mathfrak{l}_g(\mu) :=
\#\left\{
\begin{array}{c}
(\sigma_\circ,\sigma_\bullet) \in (\mathfrak{S}_n)^2 ~ ~;~ ~ \ell(\sigma_\circ)+\ell(\sigma_\bullet)+\ell(\sigma_\circ\sigma_\bullet) = n+2-2g ~ ~;\\ \langle\sigma_\circ,\sigma_\bullet\rangle \mbox{ is transitive} ~ ~ ; ~ ~  \sigma_\circ\sigma_\bullet \mbox{ has cycle type } \mu. 
\end{array}
\right\}.
$$
We now form the \emph{exponential} generating function of these numbers, where $t$ marks the number of edges and for $i\geq 1$, the variable $p_i$ marks the number of faces of degree $2i$:
\begin{align*}
L_g\equiv L_g(t;p_1,p_2,\dots):=
\mathbf{1}_{g=0} +
 \sum_{n\geq 1}\frac{t^n}{n!}
\sum_{\mu\vdash n} \mathfrak{l}_g(\mu) p_\mu,
\end{align*}
where the indicator function accounts for the unique map of genus $0$ with $1$ vertex and $0$ edge, that we allow by convention.
Similarly, for $n,k\geq 1$ and $\mu\vdash n-k$, we let $\mathfrak{b}_g(k;\mu)$ be the number of \emph{rooted} bipartite maps of genus $g$ with $n$ edges, such that the root face has half-degree $k$, and the half-degrees of non-root faces are given by the parts of $\mu$. We let $F_g\equiv F_g(t;x;p_1,p_2,\dots)$ be the corresponding \emph{ordinary} generating function:
\begin{align*}
F_g \equiv F_g(t;x;p_1,p_2,\dots) :=
\mathbf{1}_{g=0} +
 \sum_{n \geq 1} t^n \sum_{k\geq 1 \atop \mu\vdash n-k}
\mathfrak{b}_{g}(k,\mu) x^k p_\mu.
\end{align*}

Our first main result is the following theorem:
\begin{thm}[Main result -- unrooted case ($g \geq 2$)]\label{thm:mainUnrooted}
Let $z\equiv z(t)$ be the unique formal power series defined by~\eqref{eq:z}.
Moreover, define the ``variables'' $\eta$ and $\zeta$ as the following formal power series:
\begin{align*}
\eta=\sum_{k\geq 1} (k-1) {2k-1 \choose k} p_k z^k,
\quad
\zeta=\sum_{k\geq 1} \frac{k-1}{2k-1} {2k-1 \choose k} p_k z^k,
\end{align*}
and the variables $(\eta_i)_{i\geq 1}$ and $(\zeta_i)_{i\geq 1}$ by
\begin{align*}
\eta_i:=\sum_{k\geq 1} (k-1) k^i {2k-1 \choose k} p_k z^k ,
\quad
\zeta_i=\sum_{k\geq 1} \frac{(-2)^{i+1} k(k-1) \cdots (k-i)}{(2k-1)(2k-3) \cdots (2k-2i-1)}{2k-1 \choose k} p_k z^k .
\end{align*}
Then for $g \geq 2$, the exponential generating function $L_g$ of labelled bipartite maps of genus $g$ is given by a finite sum:
\begin{align}\label{eq:mainUnrooted}
L_g &= \sum_{\alpha, \beta, a, b} c_{a,b}^{\alpha, \beta} \frac{\eta_\alpha \zeta_\beta}{(1-\eta)^a (1+\zeta)^b}, 
\end{align}
for rational numbers $c_{a,b}^{\alpha, \beta}$, where the (finite) sum is taken over integer partitions $\alpha, \beta$ and nonnegative integers $a,b,$ such that 
$|\alpha| + |\beta| \leq 3(g-1)$ and $a+b = \ell(\alpha)+\ell(\beta)+2g-2$.
\end{thm}
\begin{exa}[unrooted generating function for genus $2$]\label{ex:unrootedGenus2}
\begin{multline*}
\footnotesize
L_2=
{\frac {1}{120}}
-{\frac {1}{23040}}{\frac {\eta_{{1}} \left( 185
\eta_{{1}}-58\eta_{{2}} \right) }{ \left( 1-\eta \right) ^{4}
}}
-{\frac {1}{46080}}{\frac {20\eta_{{3}}-168\eta_{{2}}+415
\eta_{{1}}}{ \left(1 - \eta \right) ^{3}}}
-\frac{53/15360}{(1-\eta)^2}
\\
-{\frac {7}{2880}}{\frac {{\eta_{
{1}}}^{3}}{ \left(1- \eta \right) ^{5}}}
-{\frac {1/512}{ \left(
1-\eta \right)  \left( 1+\zeta  \right) }}
+{\frac {
\eta_{{1}}/1536}{ \left( 1-\eta \right) ^{2} \left( 1+
\zeta  \right) }}
-{\frac {3}{1024}} \frac{1}{(1+\zeta )^{2}}
+{\frac {3}{8192}}{\frac {\zeta _{1}}{
 \left( 1+\zeta  \right) ^{3}}}.
\end{multline*}
\end{exa}
The case of genus $1$ is stated separately since it involves logarithms:
\begin{thm}[Unrooted case for genus $1$]\label{thm:unrootedGenus1}
The exponential generating function $L_1\equiv L_1(t; p_1,p_2,\dots)$ of bipartite maps on the torus is given by the following expression, with the notation of~\Cref{thm:mainRooted}:
\begin{align*}
L_1 = 
\frac{1}{24} \ln \frac{1}{1-\eta} + 
\frac{1}{8}  \ln \frac{1}{1+\zeta}.
\end{align*}
\end{thm}

In order to establish \Cref{thm:mainUnrooted} we will first prove its (slightly weaker) rooted counterpart:
\begin{thm}[Main result -- rooted case]\label{thm:mainRooted}
Let $z\equiv z(t)$ and $u=u(x,t)$ be defined by~\eqref{eq:z}--\eqref{eq:u}, and let the variables $\eta$, $\zeta$ and
$(\eta_i)_{i\geq 1}$ and $(\zeta_i)_{i\geq 1}$ be defined as in \Cref{thm:mainRooted}.
Then for all $g\geq1$, the generating function $F_g\equiv F_g(t;x;p_1,p_2,\dots)$ of rooted bipartite maps of genus~$g$ is equal to
\begin{align}\label{eq:mainRooted}
F_g 
 &= \sum_{c=1}^{6g-1} \sum_{\alpha, \beta, a \geq 0, b \geq 0} \frac{\eta_\alpha \zeta_\beta}{(1-\eta)^a (1+\zeta)^b} \left( \frac{d_{a,b,c,+}^{\alpha, \beta}}{(1-uz)^c} + \frac{d_{a,b,c,-}^{\alpha, \beta}}{(1+uz)^c} \right) , 
\end{align}
for $d_{a,b,c,\pm}^{\alpha, \beta} \in \rationals$, with the same notation as in~\Cref{thm:mainUnrooted}. Furthermore, $d_{a,b,c,\pm}^{\alpha, \beta} \neq 0$ implies $(2 \pm 1)g \geq \lceil \frac{1+c}{2} \rceil + |\alpha| + |\beta|$ and $a+b = \ell(\alpha)+\ell(\beta)+2g-1$ for the two signs, and the sum above is finite.
\end{thm} 

\noindent{\bf Some comments on the theorems.}
\normalfont
\begin{itemize}[itemsep=0pt, topsep=0pt, parsep=0pt, leftmargin=10pt]
\item[$\bullet$]Note that the ``Greek'' variables $\eta, \zeta, \eta_i, \zeta_i$ are all infinite \emph{linear} combinations of the $p_kz^k$ with explicit coefficients.
 Moreover, for fixed $g$ the sums~\eqref{eq:mainRooted}, \eqref{eq:mainUnrooted} depend only of \emph{finitely many} Greek variables, see \textit{e.g.} \Cref{ex:unrootedGenus2}.
Note also that if only finitely many $p_i$'s are non-zero, then all the Greek letters are polynomials in the \emph{unique} variable $z$. For example, if $p_i = \mathbf{1}_{i=2}$, \textit{i.e.} if we enumerate bipartite quadrangulations, all Greek variables are polynomials in the variable $s$ ($=z+1$) defined in~\Cref{eq:introTutte}. In particular, and since bipartite quadrangulations are in bijection with general rooted maps (see \textit{e.g.}~\cite{Schaeffer:survey}), the rationality results of~\cite{BC} are a (very) special case of our results.
\item[$\bullet$]
Readers familiar with the bijective techniques of map enumeration will notice that the change of variables $(t,x) \leftrightarrow (z,u)$ is very natural in view of the link between bipartite maps and  mobiles~\cite{Chapuy:constellations}. However, those bijections are still far from giving combinatorial proofs of Theorems~\ref{thm:mainRooted} and~\ref{thm:mainUnrooted}.
\item[$\bullet$]
The case of genus~$0$ is not covered by the theorems above but is well known, and we will use it thoroughly. See~\Cref{prop:F0} below.
\end{itemize}



We conclude this section with a last notation that will be useful throughout the paper. In addition to the ``Greek'' variables $\eta, \zeta, (\eta_i)_{i\geq 1}, (\zeta_i)_{i\geq 1}$ already defined, we introduce the variable $\gamma$ as the following formal power series:
\begin{align*}
\gamma:=\sum_{k\geq 1} {2k-1 \choose k} p_k z^k. 
\end{align*}
Note that the change of variables~\eqref{eq:z} relating $z$ to $t$ is given by $z=t(1+\gamma)$.

\section{The Tutte equation, and the proof strategy of \Cref{thm:mainRooted}}
\label{sec:Tutte}



\subsection{The Tutte equation}

In this section, we state the main Tutte/loop equation that is the starting point of our proofs. We first define some useful operators. The \emph{rooting operator} $\Gamma$ is defined by
\begin{align}\label{eq:Gamma}
\displaystyle
\Gamma:=\sum_{k\geq 1} k x^k \frac{\partial}{\partial p_k}.
\end{align}
Combinatorially, the effect of $\Gamma$ is to mark a face of degree $2k$,  distinguish one of its $k$ white corners, and record the size of this face using the variable $x$. In other words, $\Gamma$ is the operator that selects a root face in a map. From the discussion of~\Cref{sec:defs}, it is easy to see that $F_g = \Gamma L_g$. 

If $F\equiv F(x)$ is a f.p.s whose coefficients are polynomials in $x$ (over some ring), we let $\Delta F(x) = \frac{F(x) - F(0)}{x}$.
 Equivalently, 
$\Delta F(x) = [x^{\geq 0}] \frac{1}{x}F(x)$ where $[x^{\geq 0}]$ is the operator that selects monomials with with positive powers in $x$. 
We define the operator:
\begin{align}\label{eq:Omega}
\displaystyle\Omega := \sum_{k \geq 1} p_k \Delta^k = [x^{\geq 0}] \sum_{k \geq 1} \frac{p_k}{x^k}.
\end{align}

\begin{prop}[Tutte equation -- folklore]\label{prop:Tutte}
The sequence $(F_g)_{g\geq 0}$ of formal power series in $\mathbb{Q}[p_1,p_2,\dots][x][[t]]$ is uniquely determined by the equations, for $g\geq 0$:
\begin{align}\label{eq:Tutte}
F_g 
= 
\mathbf{1}_{g=0} 
+x t \Omega F_{g}
+x t F^{(2)}_{g-1}
+x t \sum_{g_1+g_2 = g \atop g_1,g_2\geq 0} F_{g_1} F_{g_2}
,
\end{align}
where $F^{(2)}_{g-1} := \Gamma F_{g-1}$ is the g.f. of bipartite maps of genus $g$ with \emph{two} root faces.
\end{prop}
\begin{figure}[h!!!!!]\centering
\includegraphics[width=0.8\linewidth]{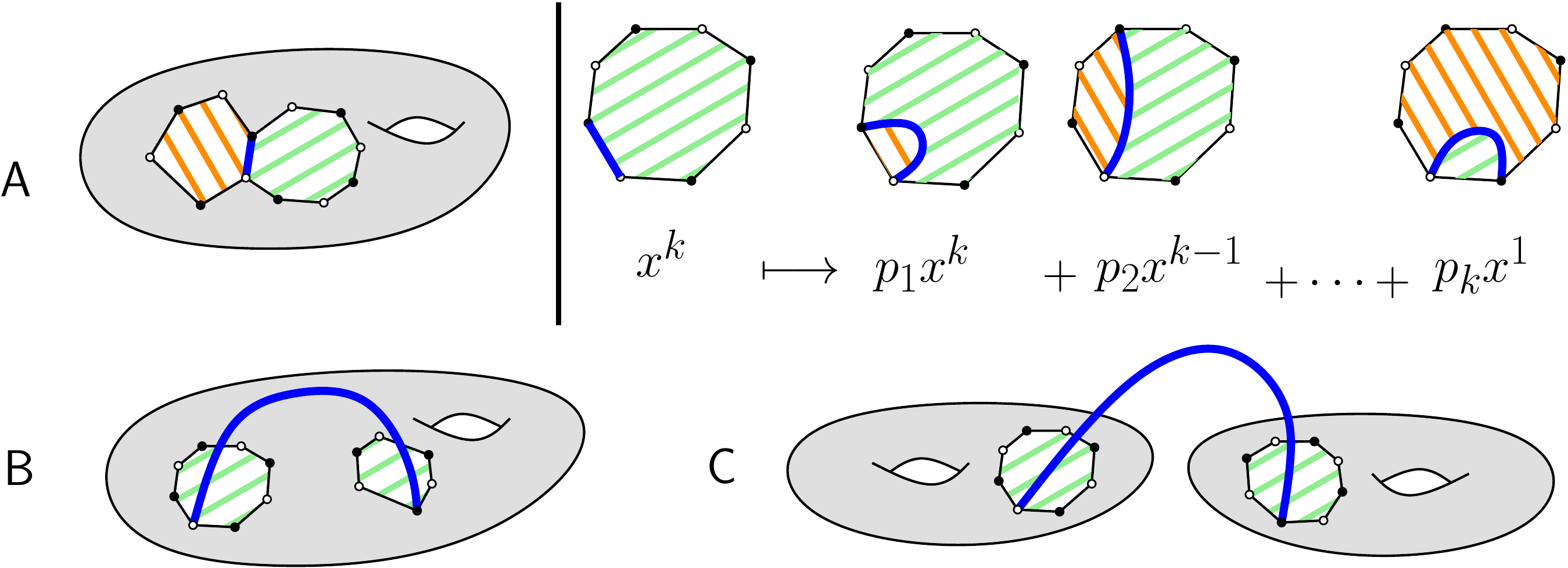}
\caption{A, B, C: Decomposition of a bipartite map by removing the root edge, leading to the three terms of the Tutte equation~\eqref{eq:Tutte}. The root face is represented in (green) rising hatching pattern, and the root edge is in (blue) fat width. Top-Right: The operator that handles case~A on generating functions.}\label{fig:proofTutte}
\end{figure}
\begin{proof}
This is equation is classical, let us briefly recall the proof. Start with a rooted map $\mathfrak{m}$ of genus $g$ with $n$ edges, and assume that $\mathfrak{m}$ has at least one edge (the case when $\mathfrak{m}$ has no edge happens only in genus $0$ and is taken into account by the indicator function). Now remove the root edge $e$ of $\mathfrak{m}$. Three things can happen (Figure~\ref{fig:proofTutte}).
\begin{itemize}[itemsep=0pt,parsep=0pt,topsep=0pt,leftmargin=14pt]
\item[A.] Removing $e$ does not disconnect $\mathfrak{m}$, \emph{and} $\mathfrak{m}$ is bordered by two different faces in $\mathfrak{m}$. In this case, removing $\mathfrak{m}$ gives rise to a new map $\mathfrak{m}'$ with one less face, one less edge, and the same vertex set, hence, by Euler's formula, genus $g$. Conversely, given any map $\mathfrak{m}'$ of genus $g$ with a root face of degree $k$, we can split the root face of $\mathfrak{m}'$ in $k$ different ways to obtain a map $\mathfrak{m}$ as above. The operator taking this operation into account on generating functions is given by:
$$
x^k \longmapsto p_1 x^{k} + p_2 x^{k-2} + \dots + p_{k}x = x \cdot \Big( \Omega x^k \Big),
$$
see Figure~\ref{fig:proofTutte}. Summing over all maps $\mathfrak{m}'$ of genus $g$, the contribution for this case is therefore $xt \Omega F_{g}$. 
\item[B.] Removing $e$ does not disconnect $\mathfrak{m}$, \emph{and} $\mathfrak{m}$ is bordered twice by the same face in $\mathfrak{m}$. In this case, removing $\mathfrak{m}$ gives rise to a new map $\mathfrak{m}'$ with one more face, one less edge, and the same vertex set, hence, by Euler's formula, genus $g-1$. Conversely, given any map $\mathfrak{m}'$ of genus $g-1$ with two root faces, merging them via a new edge reconstructs a map $\mathfrak{m}$ as above. To construct a map $\mathfrak{m}'$ of genus $g-1$ with two root faces, we start by a rooted map of genus $g-1$, counted by $F_{g-1}$, and we select a face and a corner inside it. Therefore the contribution for this second case is given by $x t \Gamma F_{g-1}$.
\item[C.] Removing $e$ disconnects $\mathfrak{m}$. We are thus left with two maps $\mathfrak{m}_1$ and $\mathfrak{m}_2$. Each of them is naturally rooted, and by Euler's formula the genera of these maps add up to $g$. Therefore this case gives rise to a contribution of $xt \sum_{g_1+g_2 = g \atop g_1,g_2\geq 0} F_{g_1} F_{g_2}. \hfill\qedhere$
\end{itemize}
\end{proof}
In genus $0$, the Tutte equation \eqref{eq:Tutte} was solved by Bender and Canfield~\cite{BC:planar} who gave the following remarkable expression in terms of the variables $z\equiv z(t)$, $u\equiv u(t;x)$ defined by~\eqref{eq:z}--\eqref{eq:u}:
\begin{prop}[Bender and Canfield~\cite{BC:planar}]\label{prop:F0}
The generating function of rooted bipartite maps of genus $0$ is given by:
\begin{align}\label{eq:F0}
F_0 = (1+uz) \left( 1 - \sum_{k=1}^{K} p_k z^k \sum_{\ell=1}^{k-1} u^\ell z^\ell \binom{2k-1}{k+\ell} \right).
\end{align}
\end{prop}

The strategy we will use to prove \Cref{thm:mainRooted} is to solve~\eqref{eq:Tutte} recursively on the genus $g$. Note that, for $g\geq 1$, and assuming that all the series $F_h, F^{(2)}_h$ are known for $h<g$, the Tutte equation \eqref{eq:Tutte} is \emph{linear} in the unknown series $F_g(x)$. More precisely it is a linear ``catalytic'' equation for the unknown series $F_g$ involving one catalytic variable (the variable $x$), see e.g.~\cite{MBM:icm}. Therefore it is tempting to solve it via the kernel method or one of its variant.

In what follows, in order to make the induction step feasible, we will need to fix an arbitrary integer $K\geq 2$, and to make the substitution $p_i=0$ for $i>K$ in \eqref{eq:Tutte}. The integer $K$ will be sent to infinity at the end of the induction step.
 To prevent a possible misunderstanding, we warn the reader that the substitution of $p_i$ to zero does \emph{not} commute with $\Gamma$, and in particular: 
$$F^{(2)}_g \Big|_{p_i =0} ~=~ \Big(\Gamma F_g\Big)\Big|_{p_i=0} 
~\neq~ \Gamma \Big(F_g\Big|_{p_i=0}\Big).
$$
In concrete terms, even after we set the variables $p_i$ to zero for all $i>K$, the series $F^{(2)}_g$ still counts maps in which the \emph{two} root faces may have arbitrarily large degrees.
We now proceed with the inductive part of the proof, that will occupy the rest of this section.
 The base case $g=1$ of the induction will be proved here as well (with empty induction hypothesis). 
To formulate our induction hypothesis, we need the following notion: if $A(u)$ is a rational function over some field containing $z$, we say that $A$ is \emph{$uz$-symmetric}  if $A(\frac1{z^2u})= A(u)$, and \emph{$uz$-antisymmetric} if $A(\frac{1}{z^2u})=- A(u)$. 

\smallskip
 \fbox{\begin{minipage}{0.9\textwidth}
{\it 
Induction Hypothesis: In the rest of~\Cref{sec:Tutte}, we fix $g\geq 1$. We assume that for all genera $g'\in[1..g-1]$, $\Cref{thm:mainRooted}$ holds for genus $g'$. In particular  $F_{g'}$ is a rational function of~$u$. We assume that it is $uz$-antisymmetric. }
\end{minipage}}

\smallskip
\noindent  We now start examining the induction step. {\bf From now on, we assume that $p_i=0$ for $i>K$.} In other words, each series mentioned below is considered under the substitution $\{p_i=0,i>K\}$, even if the notation does not make it apparent. Our first observation is the following:
\begin{prop}[Kernel form of the Tutte equation]\label{prop:TutteY}
Define 
$Y := 1- 2tx F_0 - tx \theta$, where $\theta:=\displaystyle \sum_{k=1}^K \frac{p_k}{x^k}$.
 Then one has:
\begin{align}\label{eq:TutteY}
Y F_g = 
x t F^{(2)}_{g-1}
+x t \sum_{g_1+g_2 = g \atop {g_1,g_2>0}} F_{g_1} F_{g_2}
+xt S, 
\end{align}
where $S\equiv S(t,p_1,p_2,\dots;x)$  is an element of $\mathbb{Q}[p_1,p_2,\dots][[t]]\left[\frac{1}{x}\right]$ of degree at most $K-1$ in $\frac{1}{x}$ without constant term.
\end{prop}
\begin{proof}
Consider the Tutte equation~\eqref{eq:Tutte}.
Keeping in mind the fact that we have done the substitution $p_i = 0$ for $i > K$, we observe that
\[ \Omega F_g = [x^{\geq 0}] F_g \theta, \]
where $[x^{\geq 0}]$ is the operator on $\mathbb{Q}[p_1,p_2,\dots][x^{-1},x][[t]]$ that keeps only the nonnegative powers of $x$.
Now let $S$ be the negative part of $F_g\theta$, \textit{i.e}: 
$$S := [x^{<0}] F_g \theta = F_g \theta - [x^{\geq 0}] F_g \theta. $$ 
Since $\theta$ is in $\mathbb{K}[x^{-1}]$ and of degree $K$, and since $F_g$ has no constant terms in $x$, $S$ is also in $\mathbb{K}[x^{-1}]$ and has degree at most $K-1$.
Since $\theta F_g = \Omega F_g + S$, 
we can now rewrite the equation as follows.
\[F_g = x t \theta F_{g} + xtS +x t F^{(2)}_{g-1} +x t \sum_{\substack{g_1+g_2 = g \\ g_1,g_2 > 0}} F_{g_1} F_{g_2} + 2xtF_0F_g. \]
We now move all terms involving $F_g$ to the left and factor out $F_g$, to obtain~\eqref{eq:TutteY}.
\end{proof}

\subsection{Rational structure of $F_g$ and the topological recursion}

In this section we describe in detail the structure of the kernel $Y$ and of the generating function $F_g$, in order to establish our main recurrence equation (\Cref{thm:toprec}). We leave the proofs of the most technical statements to~\Cref{sec:Y} and~\Cref{sec:Gamma}.

\newcommand{\rA}{\mathbb{A}}
\newcommand{\PP}{\mathbb{P}}
In order to analyse~\Cref{prop:TutteY} it is natural to study the properties of the ``kernel'' $Y$.
In what follows, we will consider polynomials in $\rA[z][u]$ or $\rA[[z]][u]$ where $\rA:=\mathbb{Q}(p_1,p_2,\dots,p_K)$. Note that any such polynomial, viewed as a polynomial in $u$, is split over $\PP:=\overline{\rA}((z^*))$.
 An element $u_0\in\PP$ is \emph{large} if it starts with a negative power in $z$, and is \emph{small} otherwise.  The following result is a consequence of~\eqref{eq:F0} and some computations that we delay to~\Cref{sec:Y}. As explained in \Cref{sec:notation}, it is implicit in the following that generating functions are considered under the change of variables $(t,x)\leftrightarrow (z,u)$:
\begin{prop}[Rational structure of the kernel]\label{prop:structY}
$Y$ is an element of $\mathbb{Q}(z,u,p_1,p_2,\dots,p_K)$ of the form:
$$
Y = \frac{N(u)(1-uz)}{u^{K-1}(1+\gamma) (1+uz)}
$$
 where $N(u)\in\rA[z][u]$ is a polynomial of degree $2(K-1)$ in $u$.
\end{prop} 
\begin{proof}
See \Cref{sec:Y}.
\end{proof}


\begin{prop}[Structure of zeros of the kernel] \label{prop:structY-zero}
~
\begin{itemize}
\item[(1)] $Y$ is $uz$-antisymmetric.
\item[(2)] Among the $2(K-1)$ zeros of $N(u)$ in $\PP$, $(K-1)$ are small and $(K-1)$ are large, and large and small zeros are permuted by the transformation $u\leftrightarrow \frac{1}{z^2u}$.
\end{itemize}
\end{prop}
\begin{proof}
See \Cref{sec:Y}.
\end{proof}

\noindent Before solving~\eqref{eq:TutteY}, we still need to examine more closely the structure of $F_g$.
In what follows each rational function $R(u)\in \mathbb{B}(u)$ for some field $\mathbb{B}$ is implicitely considered as an element of $\overline{\mathbb{B}}(u)$. In particular its denominator is split, and the notion of \emph{pole} is well defined (poles are elements of $\overline{\mathbb{B}}$). Moreover, $R(u)$ has a partial fraction expansion, with coefficients in $\overline{\mathbb{B}}$, and the \emph{residue} of $R(u)$ at a pole $u_* \in \overline{\mathbb{B}}$ is a well defined element of $\overline{\mathbb{B}}$, namely the coefficient of $(u-u_*)^{-1}$ in this expansion.  
The following result is perhaps the most crucial conceptual step of the topological recursion and of the proof of \Cref{thm:mainRooted}:
\begin{prop}[Structure and poles of $F_g$]\label{prop:noPoles!}
$F_g$ is an $uz$-antisymmetric element of  $\rA[[z]](u)$.
Its poles, that are elements of $\PP$, are contained in $\{\frac{1}{z},-\frac{1}{z}\}$. Moreover, $F_g$ has negative degree in $u$.
\end{prop}
The proof of~\Cref{prop:noPoles!} uses the next two lemmas:
\begin{lemma}\label{lemma:stabGamma}
If $A$ is an element of $\mathbb{Q}(u,z,\gamma,\eta,\zeta,(\eta_i)_{i \geq 1}, (\zeta_i)_{i \geq 1})$ with negative degree in $u$ whose poles in $u$ are among $\{\pm\frac{1}{z}\}$, then so is $\Gamma A(u)$. Moreover, if $A(u)$ is $uz$-antisymmetric, then $\Gamma A(x)$ is $uz$-symmetric.
\end{lemma}
\begin{proof}
See \Cref{sec:actionGamma}.
\end{proof}

\begin{lemma}\label{lemma:noSmall}
Let $A(u) \in  \mathbb{B}[[z]](u)\cap\mathbb{B}[u]((z)) \subset \mathbb{B}(u)((z))$ be a rational function in $u$ whose coefficients are formal power series in $z$ over some field $\mathbb{B}$, and that as a Laurent series in $z$ has  coefficients that are polynomials in $u$.
Then $A(u)$, seen as a rational function in $u$, has no small pole. 
\end{lemma}
\begin{proof}
By the Newton-Puiseux theorem, we can write $A(u)=\frac{P(u)}{c \cdot Q_1(u)Q_2(u)}$ with $P(u) \in \mathbb{B}[[z]][u]$, $c\in \overline{\mathbb{B}}((z^*))$, $Q_1(u) = \prod_i(1-u u_i)$ and $Q_2(u)=\prod_j (u-v_j)$, where the $u_i$, $v_j$ are \emph{small} Puiseux series over an algebraic closure $\overline{\mathbb{B}}$ of $\mathbb{B}$ and $v_j$ without constant term. Since $P(u)/Q_2(u) = c A(u)Q_1(u)$, and since $\overline{\mathbb{B}}[u]((z^{*}))$ is a ring, we see that $P(u)/Q_2(u)\in \overline{\mathbb{B}}[u]((z^{*}))$ . But since  $1/Q_2(u) = \prod_j \sum_{k\geq0}u^{-1-k}v_i^k$ is in $\overline{\mathbb{B}}[u^{-1}]((z^{*}))$, this is impossible unless $Q_2$ divides $P$ in $\overline{\mathbb{B}}((z^*))[u]$, which concludes the proof.
\end{proof}

We can now give the proof of \Cref{prop:noPoles!}:
\begin{proof}[Proof of \Cref{prop:noPoles!}]
We first claim that the R.H.S. of~\eqref{eq:TutteY} is $uz$-symmetric. In the case $g\geq 2$ this follows by induction, since each term $F_{g_1}F_{g_2}$ is $uz$-symmetric as a product of two $uz$-antisymmetric factors, the term $F^{(2)}_{g-1}$ is $uz$-symmetric using~\Cref{lemma:stabGamma}, and $S$, as any rational fraction in $x$, is symmetric since $x(u)=\frac{u}{(1+zu)^2}$ is symmetric. 
In the case $g=1$, the R.H.S. of~\eqref{eq:TutteY} is equal to $xtF_0^{(2)} + xtS$, so it is enough to see that $F_0^{(2)}$ is $uz$-symmetric. Now, the series $F_0^{(2)}$ is given by the explicit expression:
\begin{align}\label{eq:F02}
 F^{(2)}_{0}  = \frac{u^2 z^2}{(1-uz)^4}. 
\end{align}
This expression can be found in \cite{Eynard:book} (recall what \cite{Eynard:book} calls \textit{bipartite} maps do not coincide with bipartite maps in general, but they coincide in genus $0$, so we can use this result here).
It can also be obtained from direct computations from the explicit expression of $F_0$ given by \Cref{prop:F0}, and it is also easily derived from \cite[Thm. 1]{ColletFusy} (in the case $p=r=2$, with the notation of this reference). Since \eqref{eq:F02} is clearly $uz$-symmetric,  the claim is proved in all cases.

Hence by
\Cref{prop:structY-zero}, $F_g$ is $uz$-antisymmetric, being the quotient of the $uz$-symmetric right-hand side of~\eqref{eq:TutteY} by $Y$. 
Now, by the induction hypothesis and~\Cref{lemma:stabGamma} 
(or by a direct check on~\eqref{eq:F02} in the case $g=1$),
 the R.H.S. of~\eqref{eq:TutteY} is in $\rA[[z]](u)$, and its poles are contained in $\{\pm \frac{1}{z},0\}$. Hence solving~\eqref{eq:TutteY} for $F_g$ and using \Cref{prop:structY}, we deduce that $F_g$ belongs to $\rA[[z]](u)$ and that its only possible poles are $\pm\frac{1}{z}, 0$ and the zeros of $N(u)$.

Now, viewed as a series in $z$, $F_g$ is an element of $\rA[u][[z]]$. Indeed, in the variables $(t,x)$, $F_g$ belongs to $\mathbb{Q}[p_1,\dots,p_K][x][[t]]$ for clear combinatorial reasons, and as explained in \Cref{sec:notation} the change of variables $t,x \leftrightarrow z,u$ preserves the polynomiality of coefficients.
 Therefore by \Cref{lemma:noSmall}, $F_g$ has no small poles. This excludes $0$ and all small zeros of $N(u)$.
Since $F_g$ is $uz$-antisymmetric and since by \Cref{prop:structY-zero}, the transformation $z\leftrightarrow\frac{1}{z^2u}$ exchanges small and large zeros of $N(u)$, this also implies that $F_g$ has no pole at the large zeros of $N(u)$.

The last thing to do is to examine the degree of $F_g$ in $u$. We know that $S$ is a polynomial in $x^{-1}$ of degree at most $K-1$, thus has degree at most $K-1$ in $u$. 
Therefore by induction and~\Cref{lemma:stabGamma}
(or by a direct check on~\eqref{eq:F02} in the case $g=1$)
 the degree in $u$ of the R.H.S. of~\eqref{eq:TutteY} is at most $K-2$. Since the degree of $Y$ is $K-1$, the degree of $F_g$ in $u$ is at most $-1$.

\end{proof}

\begin{rem}\label{rem:oldKernel}\normalfont
Analogues of the previous proposition, stated in similar contexts~\cite[Chap. 3]{Eynard:book} play a crucial role in Eynard's ``topological recursion'' framework. 
To understand the importance of~\Cref{prop:noPoles!}, let us make a historical comparison. The ``traditional'' way of solving~\eqref{eq:TutteY} with the kernel method would be to substitute in~\eqref{eq:TutteY} \emph{all} the small roots of $N(u)$, and use the $(K-1)$ equations thus obtained to eliminate the ``unknown'' polynomial $S$. Not surprisingly, this  approach was historically the first one to be considered, see e.g.~\cite{Gao}. It leads to much weaker rationality statements than the kind of methods we use here, since the cancellations that appear between those $(K-1)$ equations are formidable and very hard to track. As we will see, \Cref{prop:noPoles!} circumvents this problem by showing that we just need to study~\eqref{eq:TutteY} at the \emph{two} points $u=\pm \frac{1}{z}$ rather than at the $(K-1)$ small roots of $N$.
\end{rem}

With \Cref{prop:noPoles!}, we can now apply one of the main idea of the topological recursion, namely that the whole object $F_g$ can be recovered from the expansion of~\eqref{eq:TutteY} at the critical points $u=\pm\frac{1}{z}$. 
In what follows, all generating functions considered are rational functions of the variable $u$ over $\rA[[z]]$. In particular, the notation $F_g(u)$ is a shorthand notation for the series $F_g(t;x;p_1,\dots,p_K)$ considered as an element of $\rA[[z]](u)$ (or even $\mathbb{Q}[p_1,p_2,\dots,p_K][[z]](u)$), \textit{i.e.} $F_g(u) := F_g(t(z),x(z,u),p_1,p_2,\dots)$.
We let $P(u)=\frac{1-uz}{1+uz}$ (the letter $P$ is for ``prefactor''). By \Cref{prop:noPoles!} the rational function $P(u)F_g(u)$ has only poles at $u=\pm\frac{1}{z}$ and has negative degree in~$u$. Therefore, if $u_0$ is some new indeterminate, we can write $P(u_0)F(u_0)$ as the sum of two residues:
\begin{align}\label{eq:residu}
P(u_0)F(u_0) = \mathrm{Res}_{u=\pm\frac{1}{z}} \frac{1}{u_0-u} P(u)F(u).
\end{align}
Note that this equality only relies on the (algebraic) fact that the sum of the residues of a rational function at all poles (including $\infty$) is equal to zero, no complex analysis is required.
Now, multiplying~\eqref{eq:TutteY} by $P(u)$, we find:
$$
P(u)F_g(u) = \frac{xtP(u)H_g(u)}{Y(u)} + \frac{xtP(u)S(x)}{Y(u)}.
$$
with $H_g(u)=F^{(2)}_{g-1}(u)+\sum_{g_1+g_2=g\atop g_1,g_2>0} F_{g_1}(u)F_{g_2}(u).$
Now observe that the second term in the right-hand side has \emph{no} pole at $u=\pm \frac{1}{z}$: indeed the factor $(1-uz)$ in $Y(u)$ simplifies thanks to the prefactor $P(u)$, and $xS(x)$ is a polynomial in $\frac{1}{x}=\frac{(1+uz)^2}{u}$. Returning to~\eqref{eq:residu} we have proved:
\begin{thm}[Topological recursion for bipartite maps]\label{thm:toprec}
The series $F_g(u_0)$ can be computed as:
\begin{align}\label{eq:toprec}
F_g(u_0) = \frac{1}{P(u_0)} \mathrm{Res}_{u=\pm\frac{1}{z}}\frac{P(u)}{u_0-u} \frac{xt}{Y(u)}\left(
F^{(2)}_{g-1}(u)+\sum_{g_1+g_2=g\atop g_1,g_2>0} F_{g_1}(u)F_{g_2}(u)
\right).
\end{align}
\end{thm}
Note that the R.H.S. of~\eqref{eq:toprec} involves only series $F_h$ for $h<g$ and the series $F_{g-1}^{(2)}$, which are covered by the induction hypothesis. This contrasts with~\eqref{eq:TutteY}, where the term $S(x)$ involves small coefficients of $F_g$.

\subsection{Proof of Theorem~\ref{thm:mainRooted}.}
\label{subsec:proofMainRooted}

In order to compute $F_g(u_0)$ from \Cref{thm:toprec}, it is sufficient to be able to compute the expansion of the rational fraction $\frac{H_g(u)}{Y(u)}$ at the points $u=\pm\frac{1}{z}$. The expansion of the product terms $F_{g_1}(u)F_{g_2}(u)$ is well covered by the induction hypothesis, so the main point will be to study the structure of the term $F_{g-1}^{(2)}(u)$, and the derivatives of $Y(u)$ at $u=\pm\frac{1}{z}$.
The first point will require to study closely the action of the operator $\Gamma$ on Greek variables, and the second one requires a specific algebraic work. 
Note also that, in order to close the induction step, we will need to take the projective limit $K\rightarrow\infty$. Therefore, we need to prove not only that the derivatives of  $\frac{H_g(u)}{Y(u)}$ at $u=\pm\frac{1}{z}$ are rational functions in the Greek variables, but also that these functions do not depend on $K$.

In the rest of this section, we apply this program and prove~\Cref{thm:mainRooted}, admitting
two intermediate results (\Cref{prop:diffY}
and~\ref{prop:Gamma-on-Greek-expr} below), whose proofs are reported to \Cref{sec:Y} and~\ref{sec:Gamma}.

The derivatives of $Y(u)$ at the critical points can be studied by explicit computations, which require some algebraic work. This is the place where we see the Greek variables appear. In \Cref{sec:devY} we will prove:
\begin{prop}[expansion of $xtP(u)/Y(u)$ at $u=\pm\frac{1}{z}$]\label{prop:diffY}
The rational function in $u$,  $\frac{xt P(u)}{Y(u)}$,  has the following formal expansions at $u=\pm\frac{1}{z}$:
\begin{align*}
\frac{xt P(u)}{Y(u)}&= \frac1{4(1-\eta)} + \sum_{\alpha, a \geq 2|\alpha|} c'''_{\alpha, a} \frac{\eta_\alpha}{(1-\eta)^{\ell(\alpha) + 1}} (1-uz)^{a} , \\
\frac{xt P(u)}{Y(u)}&= -\frac1{(1+\zeta)(1+uz)^2} + \sum_{\alpha, a \geq 2|\alpha|} c''_{\alpha, a} \frac{\zeta_\alpha}{(1+\zeta)^{\ell(\alpha) + 1}} (1+uz)^{a-2},
\end{align*}
where $c''_{\alpha,a}, c'''_{\alpha, a}$ are computable rational numbers independent of $K$.
\end{prop}
Note that the theorem above is just a formal way of collecting all the derivatives of $\frac{xt P(u)}{Y(u)}$ at $u=\pm\frac{1}{z}$, we are not interested in convergence at all here.

The next result we admit now, to be proved in~\Cref{sec:actionGamma}, details the action of the operator $\Gamma$ on Greek variables:
\begin{prop} \label{prop:Gamma-on-Greek-expr}
The operator $\Gamma$ is a derivation on $\mathbb{Q}[p_1,p_2,\dots][x][[t]]$, i.e. it satisfies $\Gamma (AB) = A \Gamma B + B \Gamma A$.
Moreover, its action on Greek variables is given by the following expressions.
\begin{align*}
\Gamma \zeta_i &= \frac{s^{-1} - s}{8(1-\eta)s^2} \left( (2i+1)\zeta_i + \sum_{j=1}^{i-1} (-1)^{j-1} \zeta_{i-j} + 4(-1)^i (1+\zeta) \right) \\
&\quad \quad + \frac1{2}(s^{-1} - s)\left( (2i+1)(s^2-1)^i +\sum_{j=1}^{i-1} (-1)^{j-1} (s^2 - 1)^{i-j} + (-1)^i \right) \\
\Gamma \zeta &= \frac{s^{-1} - s}{8(1-\eta)s^2} (\eta + \zeta) + \frac1{8}(s^{-3} - s^{-1} - 2 + 2s) \\
\Gamma \gamma &= \frac{s^{-1} - s}{4(1-\eta)s^2} (\eta + \gamma) + \frac1{4}(s^{-3} - s^{-1}) \\
\Gamma \eta_i &= \frac{s^{-1} - s}{4(1-\eta)s^2} \eta_{i+1} + \frac1{2^{i+3}} \left( (s-s^{-1}) \partial_s \right)^{i+1} (s^{-3}-3s^{-1}+2), 
\end{align*}
where $s=\frac{1-uz}{1+uz}$.
\end{prop}


Before proceeding to the full proof of \Cref{thm:mainRooted}, we first introduce two notions of degrees that will be very helpful in our proof: the Greek degree and the pole degree.
First, we let $\mathbb{G}$ be the subring of $\mathbb{Q}(\eta, \zeta, (\eta_i)_{i \geq 1}, (\zeta_i)_{i \geq 1}, uz)$ formed by polynomials in the variables $(1-\eta)^{-1}$, $(1+\zeta)^{-1}$, $(\eta_i)_{i\geq 1}$, $(\zeta_i)_{i\geq 1}$, $(1-uz)^{-1}$, $(1+uz)^{-1}$. Equivalently, we have $\displaystyle \mathbb{G}= \mathbb{Q}\left[\frac{1}{1-\eta}, \frac{1}{1+\zeta}, (\eta_i)_{i \geq 1}, (\zeta_i)_{i \geq 1},s,s^{-1}\right]$, where 
$s=\frac{1-uz}{1+uz}.$ The Greek degree and the pole degrees are defined for elements of $\mathbb{G}$. The degree of a polynomial is defined as the highest degree of a monomial with nonzero coefficient, while the degree of a monomial is defined as the product of the degrees of its factors as follows.
The \emph{Greek degree}, denoted by $\deg_\gamma$, depends only on Greek variables, \textit{i.e.} $\deg_\gamma(s)=0$, and is defined as follows: 
\[ \deg_\gamma(1 - \eta) = \deg_\gamma(1 + \zeta) = \deg_\gamma(\eta_i) = \deg_\gamma(\zeta_i) = 1 \, \mbox{for} \, i \geq 1. \]
The \emph{pole degrees} are defined for each of the two poles $u = \pm 1/z$, and are denoted by $\deg_+$ and $\deg_-$. They depend on both Greek variables and $(1 \pm uz)$ as follows.
\[ \deg_+((1-uz)^{-1}) = 1, \deg_+(\eta_i) = \deg_+(\zeta_i) = 2i \, \mbox{for} \, i \geq 1, \]
\[ \deg_-((1+uz)^{-1}) = 1, \deg_-(\eta_i) = \deg_-(\zeta_i) = 2i \, \mbox{for} \, i \geq 1. \]
We have the following proposition. 
\begin{prop} \label{prop:Gamma-degrees}
For $T \in \mathbb{G}$ which is a monomial in $(1+\zeta)^{-1}$, $(1-\eta)^{-1}$, $(1+uz)^{-1}$, $(1-uz)^{-1}$, $\eta_i$ and $\zeta_i$ for $i \geq 1$, we have that $\Gamma T$ is also in $\mathbb{G}$ and expressed as a sum of terms that are homogeneous in Greek degree, and 
\[ \deg_\gamma(\Gamma T) = \deg_\gamma(T) - 1, \; \deg_+(\Gamma T) \leq \deg_+(T) + 5, \; \deg_-(\Gamma T) \leq \deg_-(T) +1. \]
\end{prop}
\begin{proof}
Since $\Gamma$ is derivative, we have the following expression for $\Gamma T$.
\[ \Gamma T = (\Gamma uz)\frac{\partial}{\partial (uz)} T + (\Gamma \zeta) \frac{\partial}{\partial \zeta} T + (\Gamma \eta) \frac{\partial}{\partial \eta} T + \sum_{i \geq 1} (\Gamma \eta_i) \frac{\partial}{\partial \eta_i} T + \sum_{i \geq 1} (\Gamma \zeta_i) \frac{\partial}{\partial \zeta_i} T \]

With \Cref{prop:Gamma-on-Greek-expr}, we verify that $\Gamma T \in \mathbb{G}$. For the rest of the proposition, it suffices to analyse each term for each type of degree.

We will start by the Greek degree. According to \Cref{prop:Gamma-on-uz}, $\deg_\gamma(\Gamma uz) = -1$ and $\frac{\partial}{\partial(uz)}$ does not change the Greek degree. The net effect is a $-1$ on the Greek degree. For any $\nu$ that is a Greek variable, according to \Cref{prop:Gamma-on-Greek-expr}, we have $\Gamma \nu$ to be a sum of terms all of Greek degree $0$, while $\partial / \partial \nu$ decreases the Greek degree by $1$, thus the net effect is $-1$ on the Greek degree. Therefore, $\deg_\gamma(\Gamma T) = \deg_\gamma(T) - 1$.

The pole degree is more complicated. We now discuss $\deg_+$ and $\deg_-$ seperately, starting by $\deg_+$. To simplify the degree counting using expressions in \Cref{prop:Gamma-on-Greek-expr}, we recall that $s = (1-uz)/(1+uz)$, thus $\deg_+(s) = -1$ and $\deg_-(s)=1$.

We first observe that $\deg_+(\Gamma uz) = 4$ and $\partial / \partial(uz)$ will increase the pole degree $\deg_+$ by $1$, thus the net effect of this term is $5$. For terms involving Greek variables, we observe that $\deg_+(\Gamma \zeta) = 3$, and $\deg_+(\Gamma \eta) = 5$, and their corresponding differentiation does not alter the pole degree $\deg_+$, resulting in a net effect of at most $5$. For $\zeta_i$, their differentiation can decrease $\deg_+$ by $2i$ by removing a factor $\zeta_i$, but it is compensated by $\deg_+(\Gamma \zeta_i) = 2i+3$, which gives a net effect of $3$. For $\eta_i$, similarly to $\zeta_i$, their differentiation decreases $\deg_+$ by $2i$, but again $\deg_+(\Gamma \eta_i) = 2i+5$, giving a net effect of $5$. Combining all results, we have $\deg_+(\Gamma T) \leq \deg_+(T) + 5$.

We now deal with $\deg_-$. We observe that $\deg_-(\Gamma uz) = -2$ and $\partial / \partial (uz)$ increases the pole degree $\deg_-$ by $1$, and the net effect of this term is $-1$. For terms involving Greek variables, we observe that $\deg_-(\Gamma \zeta) = 1$, and $\deg_-(\Gamma \eta) = -1$, while their corresponding differentiation has no effect on $\deg_-$, and the net effect is at most an increase by $1$. For $\eta_i$ and $\zeta_i$, their differentiation decreases $\deg_-$ by $2i$ by removing a factor $\eta_i$ or $\zeta_i$, but $\deg_-(\Gamma \eta_i) = \deg_-(\Gamma \zeta_i) = 2i+1$, thus the net effect is also an increase by $1$. Therefore, $\deg_-(\Gamma T) \leq \deg_-(T) + 1$.
\end{proof}


We can now prove our first main result (up to the proofs that have been omitted in what precedes, and that will be adressed in the next sections).
\begin{proof}[Proof of \Cref{thm:mainRooted}]
We prove the theorem by induction on the genus $g \geq 1$.

We consider~\eqref{eq:toprec} in \Cref{thm:toprec}. \Cref{prop:diffY} implies that all terms in the expansion of $xtP(u)/Y(u)$ at $u=\pm z^{-1}$,  are rational fractions in the Greek variables, with denominator of the form $(1-\eta)^a (1+\zeta)^b$ for $a, b, \geq 0$. Moreover these terms do not depend on $K$ (when written in the Greek variables). 
When $g\geq 2$, 
from the induction hypothesis and
Proposition~\ref{prop:Gamma-on-Greek-expr},
 the quantity $H_g$ is a rational fraction in $u,z$ and the Greek variables, with denominator of the form $(1-\eta)^a (1+\zeta)^b (1 \pm uz)^c$ for $a,b,c\geq 0$. This rational function does not depend on $K$ (when written in the Greek variables).
 The same is true for $g=1$ using the explicit expression of $F_0^{(2)}$ given by~\eqref{eq:F02}.
Therefore, the evaluation of each residue in~\eqref{eq:toprec} is a rational function of Greek variables, independent of $K$, and with denominator of the form $(1-\eta)^a (1+\zeta)^b (1 \pm uz)^c$, with $a,b,c\geq 0$.


\smallskip

In order to prove \Cref{thm:mainRooted}, we now need to prove that $F_g$ is the sum of terms whose Greek degree $\deg_\gamma$ is at most $1-2g$, and that the pole degrees of $F_g$ verify $\deg_+(F_g) \leq 6g-1$ and $\deg_-(F_g) \leq 2g-1$. 
Note that from the induction hypothesis, for all $g'$ such that $1 \leq g' < g$, the series $F_{g'}$ verifies the degree conditions above.

We first look at $H_g$, in the case $g\geq 2$. It has two parts: the sum part, which is $\sum_{g' = 1}^{g-1} F_{g'} F_{g-g'}$, and the operator part, which is $\Gamma F_{g-1}$. We analyse the degree for both parts. For the sum part, it is easy to see that any term $T$ in the sum is homogeneously of Greek degree $\deg_\gamma(T) = 2 - 2g$, and the pole degrees verify $\deg_+(T) \leq 6g-2$ and $\deg_-(T) \leq 2g-2$. For the operator part, it results from \Cref{prop:Gamma-degrees} that $\Gamma F_{g-1}$ is a sum of terms $T$ homogeneously with Greek degree $2-2g$, and $\deg_+(\Gamma F_{g-1}) \leq 6g-2$, $\deg_-(\Gamma F_{g-1}) \leq 2g-2$. Therefore, the result from the sum part and the operator part agrees, thus $H_g$ verifies the same conditions as its two parts. 
For $g=1$, the same bound holds, as one can check from the explicit expression of $H_1=xtF_0^{(2)}$ following from~\eqref{eq:F02}.

We now observe from \Cref{prop:diffY} that all terms appearing in the expansion of $xtP/Y$ at $u\pm \frac{1}{z}$ are homogeneously of Greek degree $-1$. Therefore all the terms in the expansion of $xtPH_g/Y$ at $u=\pm \frac{1}{z}$,  have Greek degree $\deg_\gamma(H_g) + \deg_\gamma(xtP/Y) = 1-2g$. For the pole degrees,  we notice from \Cref{prop:diffY} that $\deg_+(xtP/Y) \leq 0$ and $\deg_-(xtP/Y) \leq 2$. Similar to the Greek degree, counting also the contribution from $P$, we have $\deg_+(F_g) = \deg_+(H_g) + \deg_+(xtP/Y) + 1 \leq 6g-1$ and $\deg_-(F_g) = \deg_-(H_g) + \deg_-(xtP/Y) - 1 \leq 2g-1$, and we complete the induction step.


We thus have proved that, under the specialization $p_i=0$ for $i>K$, the series $F_g$ has the form stated in \Cref{thm:mainRooted}. But, since the numbers $d_{a,b,c,\pm}^{\alpha,\beta}$ do not depend on $K$, we can let $K\rightarrow \infty$ in \eqref{eq:mainRooted} and conclude that this equality holds without considering this specialization. This concludes the proof of \Cref{thm:mainRooted}.
\end{proof}

\subsubsection*{Overview of omitted proofs.}
We have just proved \Cref{thm:mainRooted}, but we have admitted several intermediate statements in order (we hope) to make the global structure of the proof appear more clearly. All these statements will be proved in Section~\ref{sec:Gamma} and \ref{sec:Y}. In order to help the reader check that we do not forget any proof(!), we list here the statements admitted so far, and indicate  where their proofs belong:
\begin{itemize}
\item[$\bullet$] \Cref{prop:Gamma-on-Greek-expr} and \Cref{lemma:stabGamma}, that deal with the action of the operator $\Gamma$, are proved at the end of Section~\ref{sec:Gamma}. The rest of Section~\ref{sec:Gamma} contains other propositions and lemmas that prepare these proofs.
\item[$\bullet$] \Cref{prop:structY} is proved in \Cref{sec:proof-prop-structY}, where we also prove of \Cref{prop:structY-zero}.
\item[$\bullet$] \Cref{prop:diffY} is proved in \Cref{sec:devY}. This proof is rather long, especially because we choose to evaluate the generating functions with a combinatorial viewpoint, but essentially amounts to explicit computations using the explicit expression of the series $F_0$.
\end{itemize}
\noindent Therefore at the end of \Cref{sec:Gamma} and~\ref{sec:Y}, the proof of \Cref{thm:mainRooted} will be complete (without omissions). The two remaining statements (\Cref{thm:mainRooted} and \ref{thm:unrootedGenus1}) will be deduced from \Cref{thm:mainRooted} in \Cref{sec:unrooting}.

\section{Structure of the Greek variables and action of the operator $\Gamma$.}
\label{sec:Gamma}

In this section we establish several properties of the Greek variables defined in Section~\ref{sec:defsMain}. In particular we will prove Proposition~\ref{prop:Gamma-on-Greek-expr} and Lemma~\ref{lemma:stabGamma}. We also fix some notation that will be used in the rest of the paper.

\subsection{Properties of the Greek variables and their $\Theta$-images}
\label{subsec:defRingsOperators}

\newcommand{\greeks}{\mathcal{G}}
 We start by fixing some notations and defining some spaces and operators that will be used throughout the rest of the paper.
First we let $\greeks:=\{\gamma, \eta, \zeta, (\eta_i)_{i\geq 1}, (\zeta_i)_{i \geq 1}\}$ be the set of all Greek variables defined in \Cref{thm:mainUnrooted}.
Elements of $\greeks$ are infinite linear combinations of $p_k z^k$. Acting on such objects, we first define the linear operators:
\begin{align}\label{eq:defTheta}
\Theta:& ~ p_k z^k\mapsto x^k z^k,\\
D:& ~ p_k z^k \mapsto k p_k z^k.
\end{align}
Recall that the variable $z\equiv z(t;p_1,p_2,\dots)$ defined by~\eqref{eq:z} is an element of $\mathbb{Q}[p_1,p_2,\dots][[t]]$ without constant term. Therefore, each formal power series $A\in\mathbb{Q}[x,p_1,p_2,\dots][[z]]$ is an element of $\mathbb{Q}[x,p_1,p_2,\dots][[t]]$. Recall that, on this ring, the operator $\Gamma$ is defined by:
$$
\Gamma = \sum_{k\geq 1}k x^k \frac{\partial}{\partial p_k}, 
$$
where $\frac{\partial}{\partial p_k}$ is the partial differentiation with respect to $p_k$ on $\mathbb{Q}[x,p_1,p_2,\dots][[t]]$. We now introduce another operator $\partial_{p_k}$, given by the partial differentiation with respect to $p_k$ on $\mathbb{Q}[x,p_1,p_2,\dots][[z]]$ \textit{omitting the dependency of $z$ in $p_k$}. Equivalently, $\partial_{p_k}$ is defined on  $\mathbb{Q}[x,p_1,p_2,\dots][[z]]$ by the formula:
\begin{align}\label{eq:defpartialpk}
\frac{\partial}{\partial p_k} = \frac{\partial z}{\partial p_k} \frac{\partial}{\partial z} +  \partial_{p_k}.
\end{align}

Our first statement deals with the action of $\Theta$ on elements of $\greeks$. Here and later it will be convenient to work with the variable $s$ defined by:
\begin{align}\label{def:s}
s := \frac{1-uz}{1+uz}.
\end{align}
\begin{prop}
 \label{prop:Theta-base}
 The action of the operator $\Theta$ on elements of $\greeks$ is given by:
\begin{align*}
 \Theta \gamma &= \frac1{2}(s^{-1} - 1) , &
 \Theta \eta &=  \frac1{4}(s^{-3} - 3s^{-1} + 2), \\ 
 \Theta \zeta &= \frac1{4}(s+s^{-1}) - \frac1{2}, &
 \Theta \zeta_i &= (s^{-1} - s) (s^2 - 1)^i, ~ i\geq 1, \\
 \Theta \eta_i &=  \frac1{2^{i+2}} \left((s - s^{-1}) \frac{\partial}{\partial s}\right)^i (s^{-3} - 3s^{-1} + 2), ~ i\geq 1. 
\end{align*}
In particular, the images $\Theta (\eta + \gamma)$, $\Theta (\zeta - \gamma)$, $\Theta \eta_i$, $\Theta \zeta_i$ for $i \geq 1$ span the vector space $(s^{-1}-s)\mathbb{Q}[s^2, s^{-2}]$.
\end{prop}
\begin{proof}
The proof is elementary but let us sketch the computations that are not totally obvious if not performed in a good way. We observe, and will use several times, that by the Lagrange inversion formula, one has $[x^\ell]s =
 -\frac{2}{\ell} {2\ell-2 \choose \ell -1}z^\ell$ for any $\ell \geq 1$.

\noindent $\bullet$
By definition we have $\Theta \gamma = \sum_{k \geq 1} \binom{2k-1}{k} x^k z^k$, so to prove the first equality we need to show that for $k \geq 1$ one has $[x^k]\frac1{2}s^{-1}={2k-1\choose k}z^k$.
For this, we first observe by a direct computation that $s^2=1-4xz$, which implies that $2x\frac{\partial}{\partial x}s=s-s^{-1}$. It follows that 
$[x^k]\frac1{2}s^{-1}=
[x^k]\frac1{2}( s- 2x\frac{\partial}{\partial x}s)=
(1-2k)[x^k] s$, which is equal to ${2k-1 \choose k}z^k$ from the observation above.
The value of $\Theta \zeta$ is easily checked similarly, namely $[x^k](s+s^{-1})/4 = [x^k](s-x\frac{\partial}{\partial x}s)/2 = \frac{1-k}{2}[x^k]s=\frac{k-1}{2k-1}{2k-1 \choose k} z^k$.

To check the value of $\Theta \zeta_i$, we observe again that $s^2-1=-4xz$, so that $[x^k](s^{-1}-s)(s^2-1)^i=(-4z)^i[x^{k-i}](s^{-1}-s)$. 
Using again that $2x\frac{\partial}{\partial x}s=s-s^{-1}$, this is equal to $ (-4z)^i \cdot 2(i-k) [x^{k-i}]s$, which equals $(-1)^{i+1} 2^{2i+1} {2k-2i-2 \choose k-i-1} z^k$.
This quantity can be rewritten as $\frac{(-2)^{i+1}k(k-1)\dots(k-i)}{(2k-1)(2k-3)\dots(2k-2i-1)}{2k-1\choose k} z^k$ that agrees with what we expect from the definition of $\zeta_i$.

To compute $\Theta \eta$ and $\Theta \eta_i$, we first notice that 
$\Theta D = x\frac{\partial}{\partial x}\Theta$,
 and we observe that 
\[\eta = D \gamma - \gamma,\; \eta_1 = D \eta, \; \eta_i = D \eta_{i-1}. \]
We can then compute the action of $\Theta$ on these variables.
\[ \Theta \eta = \left( x\frac{\partial}{\partial x} - Id \right) \Theta \gamma = \frac1{4}(s^{-3} - 3s^{-1} + 2) \]
\[ \Theta \eta_i = \left( x\frac{\partial}{\partial x} \right)^i \Theta \eta = \frac1{2^{i+2}} \left( (s - s^{-1}) \frac{\partial}{\partial s} \right)^i (s^{-3} - 3s^{-1} + 2) \]


\noindent$\bullet$
We now prove the last statement of the proposition.
We have $\Theta (\zeta - \gamma) = (s-s^{-1})/4$ and $\Theta \zeta_i = (s^{-1} - s) (s^2 - 1)^i$ of degree $2i+1$ in $s$, and they form a triangular basis of $(s^{-1}-s)\mathbb{Q}[s^{2}]$. We also observe that $\Theta (\eta + \gamma) = (s-s^{-1}) s^{-2}/4$ and $\Theta \eta_i$ is in $(s^{-1}-s)s^{-2}\mathbb{Q}[s^{-2}]$ of degree $2i+1$ in $s^{-1}$, and also that they form a triangular basis for $(s^{-1}-s)s^{-2}\mathbb{Q}[s^{-2}]$. This proves that altogether these variables span the whole desired space.
\end{proof}

The next proposition collects some partial derivatives of our main variables that will be useful afterwards.
\begin{prop} \label{prop:diff-of-vars}
We have the following expressions of partial derivatives of the variable sets $t,x$ and $z,u$:
\begin{align*}
\frac{\partial u}{\partial x} = \frac{(1+uz)^{3}}{1-uz}, &\quad \frac{\partial z}{\partial t} = \frac{(1+\gamma)^2}{1-\eta}, \\
\frac{\partial u}{\partial t} = \frac{2(1+\gamma)^{2}u^{2}}{(1-\eta)(1-uz)}, &\quad \frac{\partial z}{\partial x} = 0 \\
\frac{\partial z}{\partial p_k} = \frac{\binom{2k-1}{k} z^{k+1}}{1-\eta}, &\quad \frac{\partial u}{\partial p_k} = \frac{2 u^2 \binom{2k-1}{k}z^{k+1}}{(1-uz)(1-\eta)} \\
\end{align*}
\end{prop}
\begin{proof}
The proof is a simple check from the definitions, via implicit differentiation. 
\end{proof}

\subsection{Action of $\Gamma$ and proofs of Proposition~\ref{prop:Gamma-on-Greek-expr} and Lemma~\ref{lemma:stabGamma}}
\label{sec:actionGamma}

We are now ready to study more explicitly the action of $\Gamma$. The next statement is obvious:
\begin{prop} \label{prop:Gamma-derivative}
The operator $\Gamma$ is a derivation, \textit{i.e.} $\Gamma (AB) = A \Gamma B + B \Gamma A$.
\end{prop}
\begin{proof}
Clear from the definition $\Gamma = \sum_{k \geq 1} k x^k \frac{\partial}{\partial p_k}$.
\end{proof}
The action of $\Gamma$ on variables $u,z,s$ can be examined by direct computation:
\begin{prop} \label{prop:Gamma-on-uz}
We have
\[ \Gamma z = \frac{z s^{-2} (s^{-1} - s)}{4(1-\eta)}, \quad \Gamma u = \frac{u s^{-2} (s^{-1} - 1) (s^{-1} - s)}{4(1-\eta)}, \quad \Gamma s = - \frac{(s^{-1} - s)^2}{8(1-\eta)s^2}  \]
\end{prop}
\begin{proof}
We proceed by direct computation by recalling the differentials computed in Proposition~\ref{prop:diff-of-vars}.
\begin{align*}
\Gamma z &= \sum_{k \geq 1} k x^k \frac{\partial}{\partial p_k} t \left( 1 + \sum_{m \geq 1} \binom{2m-1}{m} p_m z^m \right) \\
&= \sum_{k \geq 1} k \binom{2k-1}{k} x^k z^k + \frac1{1+\gamma} (\Gamma z) \sum_{k \geq 1} k \binom{2k-1}{k} p_k z^k \\
&= \frac{z}{1+\gamma} \Theta (\gamma + \eta) + \frac1{1+\gamma} (\Gamma z) (\gamma + \eta) 
\end{align*}
By solving this linear equation, we obtain $\Gamma z$.
To obtain $\Gamma u$, we notice that $\Gamma$ is a derivation and apply it to $x=u(1+uz)^2$ to obtain
\[
0 = (\Gamma u) (1+uz)^{-3} (1-uz) - (\Gamma z) 2u^2 (1+uz)^{-3}.
\]
Finally, using the fact that $\Gamma$ is derivation and the expressions of $\Gamma z$ and $\Gamma u$, we easily verify the expression of $\Gamma s$.
\end{proof}

\begin{prop} \label{prop:Gamma-on-Greek}
For $G$ a linear combination of elements of $\greeks$, we have
\[
\Gamma G = \left( \frac{s^{-1} - s}{4(1-\eta)s^2} + \Theta \right) D G
\]
\end{prop}
\begin{proof}
Since $G$ is a linear combination of Greek variables, it is an infinite linear combination of $p_k z^k$.
Recalling the definition \eqref{eq:defpartialpk} of the operator $\partial_{p_k}$, we have:
\begin{align*}
\Gamma G = \sum_{k \geq 1} k x^k \frac{\partial}{\partial p_k} G &=
 \sum_{k \geq 1} k x^k \frac{\partial z }{\partial p_k} \frac{\partial}{\partial z} G
+ 
 \sum_{k \geq 1}k x^k \partial_{p_k} G \\ &= \sum_{k \geq 1} k x^k \frac{\partial z}{\partial p_k} z^{-1} DG + \Theta D G \\
&= \left( z^{-1} (\Gamma z) + \Theta \right) DG = \left( \frac{s^{-1} - s}{4(1-\eta)s^2} + \Theta \right) DG,
\end{align*}
where the last equality uses the value of $\Gamma z$ given by the previous proposition.
\end{proof}

We are now prepared to prove \Cref{prop:Gamma-on-Greek-expr} and \Cref{lemma:stabGamma}.
\begin{proof}[Proof of \Cref{prop:Gamma-on-Greek-expr}]
The fact that $\Gamma$ is a derivation was proved in \Cref{prop:Gamma-derivative}. To obtain explicit formulas giving the action of $\Gamma$, we use \Cref{prop:Gamma-on-Greek}. For $G \in \greeks$, the value of $DG$ is given by the following list, which is straightforward from the definitions:
\begin{enumerate}
\item $D \gamma = \eta + \gamma$, $D \eta = \eta_1$, $D \zeta = \frac{\eta}{2} + \frac{\zeta}{2}$
\item $D \eta_i = \eta_{i+1}$
\item $D \zeta_i = \frac{1}{2}\left( (2i+1)\zeta_i + \sum_{j=1}^{i-1} (-1)^{j-1} \zeta_{i-j} + 4(-1)^i (\zeta + \eta) \right)$
\end{enumerate}
Since all the quantities appearing in the right-hand-side of these equalities are linear combinations of elements of $\greeks$, their images by $\Theta$ can be computed thanks to \Cref{prop:Theta-base}. Therefore using \Cref{prop:Gamma-on-Greek}, we can compute explicitly the value of $\Gamma G$ for $G\in \greeks$, and doing the algebra leads to the values given in \Cref{prop:Gamma-on-Greek-expr}.
\end{proof}



\begin{proof}[Proof of \Cref{lemma:stabGamma}]
For 
  $A \in \mathbb{Q}(u,z,\greeks)$,
 since the operator $\Gamma$ is derivative, we have the following equality.
\[ \Gamma A = (\Gamma u)\frac{\partial}{\partial u} A + (\Gamma z) \frac{\partial}{\partial z} A + (\Gamma \zeta) \frac{\partial}{\partial \zeta} A + (\Gamma \gamma)\frac{\partial}{\partial \gamma} A + (\Gamma \eta) \frac{\partial}{\partial \eta} A + \sum_{i \geq 1} (\Gamma \eta_i) \frac{\partial}{\partial \eta_i} A + \sum_{i \geq 1} (\Gamma \zeta_i) \frac{\partial}{\partial \zeta_i} A. \]
By \Cref{prop:Gamma-on-uz} and
\Cref{prop:Gamma-on-Greek-expr}, with the fact that $s=\frac{1-uz}{1+uz}$, we easily verify that $\Gamma A$ is also an element of $\mathbb{Q}(u,z,\greeks)$.
Moreover, if the poles of $A$ in $u$ are among $\pm\frac{1}{z}$, then so are the poles of $\Gamma A$. Note also that since $s$ has degree $0$ in $u$, the quantity $\Gamma G$ for $G \in \{z\} \cup \greeks$ has degree 0. Since $\Gamma u$ has degree $1$, and since differentiations decrease the degree by 1, we conclude that the degree of $\Gamma A$ is at most the degree of $A$.

We now assume that $A$ is $uz$-symmetric. For $G \in \{z\} \cup \greeks$, the operator $\frac{\partial}{\partial G}$ preserves the $uz$-antisymmetry, and according to \Cref{prop:Gamma-on-uz} and \Cref{prop:Gamma-on-Greek-expr}, $\Gamma G$ is $uz$-antisymmetric. Therefore $(\Gamma G)\frac{\partial}{\partial G}A$ is $uz$-symmetric, being the product of two $uz$-antisymmetric factors. For $u$, according to \Cref{prop:Gamma-on-uz}, $u^{-1} \Gamma u$ is $uz$-symmetric. We now inspect $\frac{u\partial}{\partial u} A$. By $uz$-antisymmetry, $A(u) = - A(u^{-1} z^{-2})$, then we have
\[ \frac{u\partial}{\partial u} A(u) = - \frac{u\partial}{\partial u} A(u^{-1} z^{-2}) = u^{-1} z^{-2} \frac{\partial A}{\partial u} (u^{-1} z^{-2}), \]
so $\frac{u\partial}{\partial u} A$ is $uz$-symmetric. Therefore, all terms in the expression of $\Gamma A$ above are $uz$-symmetric, and $\Gamma A$ is $uz$-symmetric.
\end{proof}

\section{Structure of $Y(u)$ and expansion at the critical points $u=\pm\frac{1}{z}$.}
\label{sec:Y}

In this section we study the the kernel $Y(u)$ at the points $u=\pm\frac{1}{z}$ via explicit computations. This is the place where we will see the Greek variables appear. The purpose of this section is to give the proofs of the propositions concerning $Y$, namely \Cref{prop:structY}, \Cref{prop:structY-zero} and~\Cref{prop:diffY}. 
This will conclude the proof of all auxiliary results admitted in proof of \Cref{thm:mainRooted}.

\subsection{Structure of $Y(u)$ and proof of \Cref{prop:structY}}
\label{sec:proof-prop-structY}

We can now proceed to a proof of \Cref{prop:structY} concerning the form of $Y$. 
\begin{proof}[Proof of \Cref{prop:structY}]
We can rewrite $\theta$ in the following form.
\begin{align*}
\theta &= \sum_{i=1}^{K} \frac{p_i}{x^i} = \sum_{i=1}^{K} \frac{p_i (1+uz)^{2i}}{u^i} = (1+zu) \sum_{i=1}^{K} p_i z^i \frac{(1+zu)^{2i-1}}{u^i z^i} \\
&= (1+uz) \sum_{k=1}^{K} p_k z^k \sum_{\ell=0}^{2k-1} (uz)^{\ell-k} \binom{2k-1}{\ell} \\
&= (1+uz) \sum_{k=1}^{K} p_k z^k \sum_{\ell=-k}^{k-1} u^{\ell} z^{\ell} \binom{2k-1}{k+\ell}.
\end{align*}
We recall the following expression of $F_0$ in \Cref{prop:F0}.
\[ F_0 = (1+uz) \left( 1 - \sum_{k=1}^{K} p_k z^k \sum_{\ell=1}^{k-1} u^\ell z^\ell \binom{2k-1}{k+\ell} \right). \]

We can now compute $2F_0 + \theta$ directly.
\begin{align*}
2F_0 + \theta &= (1+uz) \left( 2 - \sum_{k=1}^{K} p_k z^k \left( 2\sum_{\ell=1}^{k-1} u^\ell z^\ell \binom{2k-1}{k+\ell} - \sum_{\ell=-k}^{k-1} u^\ell z^\ell \binom{2k-1}{k+\ell} \right) \right) \\
&= (1+uz) \left( 2 - \sum_{k=1}^{K} p_k z^k \left( \sum_{\ell=1}^{k-1} u^\ell z^\ell \binom{2k-1}{k+\ell} - \sum_{\ell=-k}^{0} u^\ell z^\ell \binom{2k-1}{k+\ell} \right) \right).
\end{align*}
We observe that $u^K (2F_0 + \theta) = (1+uz) Q(u)$ with $Q(u)$ polynomial in $u$ of degree $2K-1$. The polynomial $Q(u)$ has the additional property that $[u^k] Q(u)$ is a polynomial in $z$, and for $k \geq K-1$, $[u^k] Q(u)$ is divisible by $z^{2(k-K)+1}$.
We now evaluate $2F_0 + \theta$ at the point $u=1/z$.
\begin{align*}
(2F_0 + \theta)\Big|_{u=\frac{1}{z}} &= 2 \left( 2 - \sum_{k=1}^{K} p_k z^k \left( \sum_{\ell=1}^{k-1} \binom{2k-1}{k+\ell} - \sum_{\ell=-k}^{0} \binom{2k-1}{k+\ell} \right) \right) \\
&= 2 \left( 2 + 2 \sum_{k=1}^{K} p_k z^k \binom{2k-1}{k} \right). \\
&= 4 + 4\gamma
\end{align*}
Therefore $Q\Big|_{u=\frac{1}{z}} = (2 + 2\gamma) z^{-K}$. We now write
\[ Y = 1 - xt(2F_0 + \theta) = \frac{(1+uz)^2 (1+\gamma) - uz(2F_0 + \theta)}{(1+uz)^2 (1+\gamma)}, \]
so that
\[ (1+uz) (1+\gamma) u^{K-1} Y = (1+uz) (1+\gamma) u^{K-1} - z Q(u). \]
When evaluated at $u=1/z$, the right-hand side vanishes. This proves that the left-hand side, which is a polynomial in $u$ of degree $2K-1$, has $(1-uz)$ as factor. We can thus write:
\[ Y = \frac{N(u) (1-uz)}{u^{K-1}(1+uz)(1+\gamma)} \]
with $N(u)$ polynomial in $u$ of degree $2(K-1)$.
\end{proof}
\begin{proof}[Proof of \Cref{prop:structY-zero}]
We first observe that $Y^2$ is $uz$-symmetric. Indeed (using an idea already used in \cite{BC:planar} and sometimes called the \emph{quadratic method}, see e.g.~\cite{MBM:icm}), we can rewrite the Tutte equation~\eqref{eq:Tutte} for $g=0$ as follows:
\[ (1 - xt(2F_0 + \theta))^2 = x^2 t^2 \theta^2 - 4xt - 2xt \theta + 1 - 4xt(\Omega F_0 - \theta F_0). \]
The right-hand is a Laurent polynomial in $x$, therefore it is symmetric. Since $Y=1-xt(2F_0+\theta)$, we conclude that $Y^2$ is symmetric. Now, since $Y$ is a Laurent polynomial in $zu$, it follows that $Y$ is either symmetric or antisymmetric (indeed $Y(u)^2-Y^2(\frac{1}{z^2u})$ is null and factors as $(Y(u)-Y(\frac{1}{z^2u}))(Y(u)+Y(\frac{1}{z^2u}))$, so one of the two factors must be null, as a Laurent polynomial).
To determine whether $Y$ is symmetric or antisymmetric, we examine its poles at $zu=0$ and $zu=\infty$. From the expression $Y=1-xt(2F_0+\theta)$, from the definition of $\theta$, and from the explicit expression of $F_0$ given by \Cref{prop:F0}, it is straightforward to check that:
$$
Y(u) \sim - t p_k /(zu)^{k-1} \mbox{ when } zu\rightarrow 0 \; , \;
Y(u) \sim  t p_k (zu)^{k-1} \mbox{ when } zu\rightarrow \infty.
$$
We conclude that $Y$ is antisymmetric.

 Now we study the zeros of $N(u)$. We will do this by studing the \emph{Newton polygon} of $N(u)$, defined as the convex hull of the points $(i,j)\in\mathbb{R}^2$ such that the monomial $u^iz^j$  has non zero coefficient in $N(u)$. 

We will rely on the computations done in the previous proof.
We first observe that $[u^{K-1}]((1+uz) (1+\gamma) u^{K-1} - z Q(u))$ is a polynomial in $z$ with a constant term $1$, therefore the same holds for $[u^{K-1}] N(u)$, which implies that the point $B = (K-1, 0)$ is present in the Newton polygon of $N(u)$.
Moreover, we observe that $[u^{0}]((1+uz) (1+\gamma) u^{K-1} - z Q(u)) = - [u^0] z Q(u)$. But $[u^0] Q(u) = p_K$, therefore the point $A = (0,1)$ is present in the Newton polygon of $N(u)$. For any $k < K-1$, since $[u^k] Q(u)$ is a polynomial in $z$, the point $(k,0)$ is never in the Newton polygon of $N(u)$. Therefore, the segment $AB$ is a side of the Newton polygon of $N(u)$, and accounts for the $(K-1)$ small roots of $N(u)$.

We then observe that $[u^{2K-1}]((1+uz) (1+\gamma) u^{K-1} - z Q(u)) = -z[u^{2K-1}]Q(u) = p_K z^{2K}$. Therefore, the point $C = (2(K-1), 2K-1)$ is present in the Newton polygon of $N(u)$. Furthermore, for any $ k > K-1$, $[u^{k}]((1+uz) (1+\gamma) u^{K-1} - z Q(u)) = -z[u^{k}]Q(u)$, and $[u^k]Q(u)$ is divisible by $z^{k-K+1}$, thus $[u^k]N(u)$ is divisible by $z^{2(k-K)+2}$. The point corresponding to this term is $(k, 2(k-K)+2)$, and it always above the segment $BC$. We conclude that $BC$ is a side of the Newton polygon of $N(u)$, which accounts for the $(K-1)$ large roots of $N(u)$.

It remains to prove that the transformation $u \to \frac1{uz^2}$ exchanges large and small zeros of $N(u)$. Let $u_0$ be a small zero of $N(u)$, it is also a zero of $Y(u)$. But $Y$ is $uz$-antisymmetric, therefore $Y(u_0) = Y(u_0^{-1} z^{-2})$, thus $u_0^{-1} z^{-2}$ is also a zero of $Y(u)$, and it is clearly not $1/z$. The only possibility is that $u_0^{-1} z^{-2}$ is a zero of $N(u)$, and it is a large zero. Since the transformation $u \leftrightarrow u^{-1} z^{-2}$ is involutive, we conclude that it exchanges small and large zeros of $N(u)$.

\end{proof}

\subsection{Expansion of $Y(u)$ and proof of \Cref{prop:diffY}}
\label{sec:devY}

We now study the expansion of $Y(u)$ at critical points. This is where (finally!) Greek variables appear, and what explains their presence in Theorem~\ref{thm:mainRooted}. 

We will start by the Taylor expansion of $2F_0 + \theta$. Since we are computing the Taylor expansion by successive differentiation by $u$, for simplicity, we will use the shorthand $\partial_u$ for $\frac{\partial}{\partial u}$. For integers $\ell$ and $a$, we define the \emph{falling factorial} $(\ell)_{(a)}$ to be $(\ell)_{(a)} = \ell (\ell - 1) \dots (\ell -a+1)$.

\begin{prop} \label{prop:taylor-pos-pre}
At $u=1/z$, we have the following Taylor expansion of $2F_0 + \theta$.
\[ 2F_0+\theta = 4+4\gamma - 2(1-\eta)(1-uz) + \sum_{a \geq 2} (1-uz)^a \left( (\eta + \gamma) + \sum_{i=1}^{\lfloor \frac{a-1}{2} \rfloor} c^+_{i,a} \eta_i \right) \]
Here $c_{i,a}^+$ and $d_{a}^+$ are rational numbers depending only on $i, a$.
\end{prop}
\begin{proof}
We proceed by computing successive derivatives evaluated at $u=1/z$. In the proof of Proposition~\ref{prop:structY}, we already showed that $(2F_0 + \theta)(u=1/z) = 4 + 4\gamma$, which accounts for the first term.

For other terms, we recall the expression of $2F_0 + \theta$ we used in the proof of Proposition~\ref{prop:structY}.
\begin{align*}
2F_0 + \theta &= (1+uz) \left( 2 - \sum_{k=1}^{K} p_k z^k \left( \sum_{\ell=1}^{k-1} u^\ell z^\ell \binom{2k-1}{k+\ell} - \sum_{\ell=-k}^{0} u^\ell z^\ell \binom{2k-1}{k+\ell} \right) \right) \\
&= (2+2uz) - \sum_{k=1}^{K} p_k z^k \left( \sum_{\ell=1}^{k-1} u^\ell z^\ell \binom{2k-1}{k+\ell} + \sum_{\ell=2}^{k} u^\ell z^\ell \binom{2k-1}{k+\ell-1} \right) \\
&\quad + \sum_{k=1}^{K} p_k z^k \left( \sum_{\ell=-k}^{0} u^\ell z^\ell \binom{2k-1}{k+\ell} + \sum_{\ell=-k+1}^{1} u^\ell z^\ell \binom{2k-1}{k+\ell-1} \right) \\
&=  (2+2uz) + \sum_{k=1}^{K} p_k z^k \left(  \sum_{\ell=-k}^{0} u^\ell z^\ell \binom{2k}{k+\ell} - \sum_{\ell=2}^{k} u^\ell z^\ell \binom{2k}{k+\ell} + \frac{2}{k+1} \binom{2k-1}{k} uz \right),
\end{align*}
by grouping the powers of $(uz)$ together. Now we compute the first term.
\begin{align*}
\left. \partial_u (2F_0 + \theta) \right|_{u=1/z} &= 2z + z \sum_{k=1}^{K} p_k z^k \left(  \sum_{\ell=-k}^{0} \ell \binom{2k}{k+\ell} - \sum_{\ell=2}^{k} \ell \binom{2k}{k+\ell} + \frac{2}{k+1} \binom{2k-1}{k} \right) \\
&= 2z - z \sum_{k=1}^{K} p_k z^k \left(  \sum_{\ell=0}^{k} \ell \binom{2k}{k+\ell} + \sum_{\ell=2}^{k} \ell \binom{2k}{k+\ell} + \frac{2}{k+1} \binom{2k-1}{k} \right) \\
&= 2z - z \sum_{k=1}^{K} p_k z^k (2k-2) \binom{2k-1}{k} = 2z(1-\eta)
\end{align*}

For any $a \geq 2$, the $a$-th differentiation of $2F_0 + \theta$ evaluated at $u=1/z$ is
\begin{align*}
\left. \partial_u^a (2F_0 + \theta)\right|_{u=1/z} &= z^a \sum_{k=1}^{K} p_k z^k \left(  \sum_{\ell=-k}^{0} (\ell)_{(a)} \binom{2k}{k+\ell} - \sum_{\ell=2}^{k} (\ell)_{(a)} \binom{2k}{k+\ell} \right) \\
&= z^a \sum_{k=1}^{K} p_k z^{k} \left( \sum_{\ell=1}^{k} (-1)^a \binom{2k}{k+\ell} (\ell+a-1)_{(a)} - \sum_{\ell=1}^{k} \binom{2k}{k+\ell} (\ell)_{(a)} \right)
\end{align*}

We first compute the quantity $\sum_{\ell=1}^{k} (\ell)_{(a)} \binom{2k}{k+\ell}$ given $a \geq 2$ fixed for any $k$. It is natural to consider the following generating series:
\[ D_a(y) = \sum_{k \geq 0} y^k \sum_{\ell=1}^{k} (\ell)_{(a)} \binom{2k}{k+\ell} = a! \sum_{k \geq 0} y^k \sum_{\ell=1}^{k} \binom{\ell}{a} \binom{2k}{k+\ell}. \]
We choose to compute $D_a$ via a combinatorial interpretation in terms of lattice paths.
Note that the number $[y^k]D_a/a!$ counts paths of length $2k$ with $+1$ and $-1$ steps, ending at height $2\ell$ ($k+\ell$ steps up and $k-\ell$ steps down), with $a$ distinct even and positive heights (including 0) below $2\ell$ marked. By decomposing the whole path at the last passage for each height, we have the following equality.
\[ D_a(y) = a! E(y)(1+C(y))C(y)^a \]
Here, $E(y)$ is the series of paths ending at $0$, and $C(y)$ is the series of 
paths of even length ending in a strictly positive height. All these series are classically expressed in terms of the series of Dyck paths as follows. Let $B(y)$ be the series of Dyck paths, \textit{i.e.} paths ending at $0$ and remaning always nonnegative. We have by classical decompositions $E(y) = \frac{2}{1-(B(y)-1)} - 1$ and $C(y) = \frac{y B(y)^2}{1-y B(y)^2}$. But we know that $B(y)$ verifies the equation $B(y) = 1 + y B(y)^2$, so we finally obtain the wanted expression of $D_a$:
\[ B(y) = \frac{1-\sqrt{1-4y}}{2y}, \; E(y) = \frac1{\sqrt{1-4y}}, \; C(y) = \frac1{2}\left( \frac1{\sqrt{1-4y}} - 1 \right), \]
\[ D_a(y) = \frac{a!}{2^{a+1}} \frac1{\sqrt{1-4y}} \left( \frac1{1-4y} - 1 \right) \left( \frac1{\sqrt{1-4y}} - 1 \right)^{a-1}.  \]

We now want to compute the quantity $\sum_{\ell=1}^{k} (\ell+a-1)_{(a)} \binom{2k}{k+\ell}$ given $a \geq 2$ fixed for any $k$. We consider the following generating sequence.

\[ T_a(y) = \sum_{k \geq 0} y^k \sum_{\ell=1}^{k} (\ell+a-1)_{(a)} \binom{2k}{k+\ell} = a! \sum_{k \geq 0} y^k \sum_{\ell=1}^{k} \binom{\ell+a-1}{a} \binom{2k}{k+\ell} \]
The combinatorial interpretation is essentially the same as $D_a(y)$, but in this case the $c$ heights are not necessarily distinct, therefore we have the following equality.
\[ T_a(y) = a! E(y)(1+C(y))^{a}C(y) = \frac{a!}{2^{a+1}} \frac1{\sqrt{1-4y}} \left( \frac1{1-4y} - 1 \right) \left( \frac1{\sqrt{1-4y}} + 1 \right)^{a-1} \]

Since $[p_k z^{k+a}]\left. \partial_u^a (2F_0 + \theta)\right|_{u=1/z} = [y^k]((-1)^a T_a(y) - D_a(y))$, we now consider $(-1)^a T_a(y) - D_a(y)$.

\[ (-1)^a T_a(y) - D_a(y) = \frac{a!(-1)^a}{2^{a+1}} \frac1{\sqrt{1-4y}} \frac{4y}{1-4y} \left( \left( \frac1{\sqrt{1-4y}} + 1 \right)^{a-1} + \left( 1 - \frac1{\sqrt{1-4y}} \right)^{a-1} \right) \]


We observe that, whenever $a$ is even or odd, when viewed as a polynomial in $\frac1{\sqrt{1-4y}}$, $(-1)^a T_a(y) - D_a(y)$ is always a linear combination of terms of the form $\frac{4y}{(1-4y)^{3/2}} (1-4y)^t$, and we also observe that $[y^k]\frac{4y}{(1-4y)^{3/2}} (1-4y)^t = [x^k z^k] \frac{4xz}{(1-4xz)^{3/2}} (1-4xz)^t$. We thus have the following expression of $\left. \partial_u^a (2F_0 + \theta)\right|_{u=1/z}$.
\begin{align*}
\left. \partial_u^a (2F_0 + \theta)\right|_{u=1/z} &= z^a \sum_{k = 1}^{K} p_k z^k [y^k] \left( \frac{a!(-1)^a}{2^{a+1}} \frac1{\sqrt{1-4y}} \frac{4y}{1-4y} \left( \left( \frac1{\sqrt{1-4y}} + 1 \right)^{a-1} + \left( 1 - \frac1{\sqrt{1-4y}} \right)^{a-1} \right) \right) \\
&= z^a \Theta^{-1} \left( \frac{a!(-1)^a}{2^{a}} s^{-2} (s^{-1} - s) \sum_{i=0}^{\lfloor \frac{a-1}{2} \rfloor} \binom{a-1}{2i} s^{-2i} \right)
\end{align*}

We observe that $\Theta \eta_1 = \frac{8}{3} s^{-3} (s^{-1} - s)^2$, and since $\Theta \eta_{i+1} = (s - s^{-1}) \partial_s \Theta \eta_{i}$, by induction on $i$ we know that $\Theta \eta_i$, as a Laurent polynomial in $s$, has a factor $(s-s^{-1})^2$ for $i \geq 1$. Therefore, from \Cref{prop:Theta-base} we know that, for any polynomial $P(s^{-2})$ in $s^{-2}$, $\Theta^{-1} \left( s^{-2}(s - s^{-1})P(s^{-2}) \right)$ is a linear combination of $(\eta+\gamma)$ and $\eta_i$ for $i \geq 0$, and $[\eta+\gamma]\Theta^{-1} \left( s^{-2}(s - s^{-1})P(s^{-2}) \right) = 4P(1)$ by the fact that $\Theta(\eta+\gamma)=s^{-2}(s - s^{-1})/4$. Therefore we have
\begin{align*}
\left. \partial_u^a (2F_0 + \theta)\right|_{u=1/z} &= a! z^a (-1)^a \left( (\eta + \gamma) + \sum_{i=1}^{\lfloor \frac{a-1}{2} \rfloor} c^+_{i,a} \eta_i \right),
\end{align*}
for some rational number $c^+_{i,a}$ which concludes the proof.
\end{proof}

We now perform a very similar computation for the other pole $u=-1/z$.

\begin{prop} \label{prop:taylor-neg-pre}
At $u=-1/z$, we have the following Taylor expansion of $2F_0 + \theta$.
\[ 2F_0+\theta = 2(1+\zeta)(1+uz) +  \sum_{a \geq 2} (1+uz)^a \left( (\zeta - \gamma) + \sum_{i=1}^{\lfloor \frac{a-1}{2} \rfloor} c^-_{i,a} \zeta_i \right)  \]
Here $c_{i,a}^-$ are rational numbers depending only on $i, a$.
\end{prop}
\begin{proof}
For the constant term, we first recall the following expression in the proof of \Cref{prop:structY}.
\[
2F_0 + \theta = (1+uz) \left( 2 - \sum_{k=1}^{K} p_k z^k \left( \sum_{\ell=1}^{k-1} u^\ell z^\ell \binom{2k-1}{k+\ell} - \sum_{\ell=-k}^{0} u^\ell z^\ell \binom{2k-1}{k+\ell} \right) \right)
\]
It is obvious that $\left. (2F_0 + \theta) \right|_{u=-1/z}$ vanishes.

For other terms, we will recycle the following expression of $2F_0 + \theta$ in the proof of \Cref{prop:taylor-pos-pre}.
\[
2F_0 + \theta = (2+2uz) + \sum_{k=1}^{K} p_k z^k \left(  \sum_{\ell=-k}^{0} u^\ell z^\ell \binom{2k}{k+\ell} - \sum_{\ell=2}^{k} u^\ell z^\ell \binom{2k}{k+\ell} + \frac{2}{k+1} \binom{2k-1}{k} uz \right)
\]
The first-order differentiation becomes
\begin{align*}
\left. \partial_u (2F_0 + \theta) \right|_{u=-1/z} &= 2z - z \sum_{k=1}^{K} p_k z^k \left(  \sum_{\ell=-k}^{0} (-1)^{\ell} \ell \binom{2k}{k+\ell} - \sum_{\ell=2}^{k} (-1)^{\ell} \ell \binom{2k}{k+\ell} - \frac{2}{k+1} \binom{2k-1}{k} \right) \\
&= 2z - z\sum_{k=1}^{K} p_k z^k \frac{2-2k}{2k-1} \binom{2k-1}{k} = 2z(1+\zeta)
\end{align*}

For any $a \geq 2$, the $a$-th differentiation of $2F_0 + \theta$ evaluated at $u=-1/z$ is
\begin{align*}
\left. \partial_u^a (2F_0 + \theta)\right|_{u=-1/z} &= z^a \sum_{k=1}^{K} p_k z^{k} \left( \sum_{\ell=1}^{k} (-1)^\ell \binom{2k}{k+\ell} (\ell+a-1)_{(a)} - \sum_{\ell=1}^{k} (-1)^{\ell - a} \binom{2k}{k+\ell} (\ell)_{(a)} \right)
\end{align*}

We will now much borrow the combinatorial interpretation presented in the proof of \Cref{prop:taylor-pos-pre}. We now consider the following generating functions.

\[ \tilde{D}_a(y) = \sum_{k \geq 0} y^k \sum_{\ell=1}^{k} (-1)^\ell (\ell)_{(a)} \binom{2k}{k+\ell} = a! \sum_{k \geq 0} y^k \sum_{\ell=1}^{k} (-1)^\ell \binom{\ell}{a} \binom{2k}{k+l} \]

\[ \tilde{T}_a(y) = \sum_{k \geq 0} y^k \sum_{\ell=1}^{k} (-1)^\ell (\ell+a-1)_{(a)} \binom{2k}{k+\ell} = a! \sum_{k \geq 0} y^k \sum_{\ell=1}^{k} (-1)^\ell \binom{\ell+a-1}{a} \binom{2k}{k+\ell} \]

We can see that $[p_k z^{k+a}] \left. \partial_u^a (2F_0 + \theta) \right|_{u=-1/z} = [y^k](\tilde{T}_a(y) - (-1)^a \tilde{D}_a(y))$. Furthermore, these two series have combinatorial interpretation similar with $D_a(y)$ and $T_a(y)$ in the proof of \Cref{prop:taylor-pos-pre}, with the only difference that the parity of the height at the end also contributes as a sign. We define $\tilde{C}(y) = \frac{-y B(y)^2}{1+y B(y)^2}$. We have the following equalities, with $C(y)$ and $E(y)$ borrowed from the proof of \Cref{prop:taylor-pos-pre}.

\[ \tilde{C}(y) = \frac1{2}\left( \sqrt{1-4y} - 1 \right) \]
\[ \tilde{D}_a(y) = a! E(y) (1+\tilde{C}(y)) C(y)^a = \frac{a!}{2^{a+1}} \frac{-4y}{\sqrt{1-4y}} \left( \sqrt{1-4y} - 1 \right)^{a-1} \]
\[ \tilde{T}_a(y) = a! E(y) (1+\tilde{C}(y))^a C(y) = \frac{a!}{2^{a+1}} \frac{-4y}{\sqrt{1-4y}} \left( \sqrt{1-4y} + 1 \right)^{a-1} \]

Therefore, we have

\[ \tilde{T}_a(y) - (-1)^a \tilde{D}_a(y) = \frac{a!}{2^{a+1}} \frac{-4y}{\sqrt{1-4y}} \left( \left( 1 + \sqrt{1-4y} \right)^{a-1} + \left( 1 - \sqrt{1-4y} \right)^{a-1} \right). \]


We observe that for any value of $a \geq 2$, $\tilde{T}_a(y) - (-1)^a \tilde{D}_a(y)$ is a linear combination of terms of the form $\frac{-4y}{\sqrt{1-4y}} (1-4y)^t$, and we also observe that $[y^k]\frac{-4y}{\sqrt{1-4y}} (1-4y)^t = [x^k z^k] \frac{-4xz}{\sqrt{1-4xz}} (1-4xz)^t$. We thus have the following expression of $\left. \partial_u^a (2F_0 + \theta)\right|_{u=-1/z}$.
\begin{align*}
\left. \partial_u^a (2F_0 + \theta)\right|_{u=-1/z} &= z^a \sum_{k=1}^{K} p_k z^{k} [y^k] \left( \frac{a!}{2^{a+1}} \frac{-4y}{\sqrt{1-4y}} \left( \left( 1 + \sqrt{1-4y} \right)^{a-1} + \left( 1 - \sqrt{1-4y} \right)^{a-1} \right) \right) \\
&= \frac{z^a a!}{2^{a}} \Theta^{-1} \left( (s-s^{-1}) \sum_{i=0}^{\lfloor \frac{a-1}{2} \rfloor} \binom{a-1}{2i} s^{2i} \right) \\
\end{align*}

We observe that $\Theta (\zeta - \gamma) = (s - s^{-1})/4$, and $\Theta \zeta_i = (s^{-1} - s)(s^2 - 1)^i$, therefore, for any polynomial $P$, $\Theta^{-1} ((s - s^{-1}) P(s^2))$ is a linear combination of $\zeta - \gamma$ and $\zeta_i$, and $[\zeta - \gamma]\Theta^{-1} ((s - s^{-1}) P(s^2)) = 4P(1)$. Therefore we have
\[ 
\left. \partial_u^a (2F_0 + \theta)\right|_{u=-1/z} = z^a a! \left( (\zeta - \gamma) + \sum_{i=1}^{\lfloor \frac{a-1}{2} \rfloor} c^-_{i,a} \zeta_i \right)
\]
for some rational numbers $c^+_{i,a}$.
\end{proof}

With \Cref{prop:taylor-pos-pre} and \Cref{prop:taylor-neg-pre}, we can prove \Cref{prop:diffY} now.

\begin{proof}[Proof of \Cref{prop:diffY}]
We will first rewrite $xtP/Y$ as follows.
\[ \frac{xt P}{Y} = \frac{1-uz}{1+uz} \frac1{(1+\gamma)\frac{(1+uz)^2}{uz} - (2F_0 + \theta)} \]
And now we will substitute the Taylor expansion of $2F_0 + \theta$ at $u = \pm 1/z$ into the formula above to obtain the Taylor expansion of $xtP/Y$ at corresponding points.

We will first treat the point $u=1/z$.
\begin{align*}
\frac{xtP}{Y} &= \frac{1-uz}{(8 - 4(1-uz) + 2(1-uz)^2 + \sum_{i \geq 3} (1-uz)^i) (1+\gamma) - (2 - (1-uz)) (2F_0 + \theta)} \\
&=  \frac{1}{4(1-\eta) - \sum_{a \geq 2} (1-uz)^a \left( (\eta - 1) + \sum_{i = 1}^{\lfloor a/2 \rfloor} (2c_{i,a+1}^+ - c_{i,a}^+ ) \eta_i  \right)} \\
&= \frac1{4(1-\eta)} + \sum_{\alpha, a \geq 2|\alpha|} c'''_{\alpha, a} \frac{\eta_\alpha}{(1-\eta)^{\ell(\alpha) + 1}} (1-uz)^{a}
\end{align*}

The treatment for $u=-1/z$ is similar.
\begin{align*}
\frac{xtP}{Y} &= \frac{2-(1+uz)}{ - \sum_{i \geq 3} (1+uz)^i (1+ \gamma) - (1+uz)(2F_0 + \theta)} \\
&= \frac{-2+(1-uz)}{(1+uz)^2 \left[ 2(1+\zeta) +  \sum_{a \geq 2} (1+uz)^{a-1} \left( (1 + \zeta) +  \sum_{i=1}^{\lfloor \frac{a-1}{2} \rfloor} c^-_{i,a} \zeta_i \right) \right]} \\
&= -\frac1{(1+\zeta)(1+uz)^2} + \sum_{\alpha, a \geq 2|\alpha|} c''_{\alpha, a} \frac{\zeta_\alpha}{(1+\zeta)^{\ell(\alpha) + 1}} (1+uz)^{a-2}
\end{align*}
\end{proof}

At this point, we have finished the proof \Cref{thm:mainRooted} (including all the statements that had been admitted in \Cref{sec:defsMain}). It remains to prove \Cref{thm:mainUnrooted} and \Cref{thm:unrootedGenus1}, which will be the purpose of the next section.

\section{Unrooting step and proof of Theorems~\ref{thm:mainUnrooted} and~\ref{thm:unrootedGenus1}}
\label{sec:unrooting}

In this section, we deduce \Cref{thm:mainUnrooted} from \Cref{thm:mainRooted},
 and we also check the exceptional case of genus~$1$ given by \Cref{thm:unrootedGenus1}.
Since the series $L_g$ and $F_g$ are related by the formula $F_g = \Gamma L_g$, studying $L_g$ from $L_g$ essentially amounts to inverting the differential operator $\Gamma$, \textit{i.e.}, heuristically, to perform some kind of \emph{integration}. Since in our case the generating functions of rooted maps given by \Cref{thm:mainRooted} are rational in our given set of parameters, it is no surprise that an important part of the work will be to show that this integration gives rise to no logarithm. This section is divided in two steps: we first construct (\Cref{sec:unrooting1}) two operators that enable us to ``partially'' invert the operator $\Gamma$ (\Cref{prop:decomposeGamma}),
and we reduce the inversion of the operator $\Gamma$ to the computation of a univariate integral. Then (\Cref{sec:unrooting2}) we conclude the proof of \Cref{thm:mainUnrooted} by proving that this integral contains no logarithms, using the combination of two combinatorial arguments: a disymmetry-type theorem, and a algebraicity statement proved with bijective tools in~\cite{Chapuy:constellations}.

\subsection{The operators $\Diamond$ and $\Box$.}
\label{sec:unrooting1}
\newcommand{\rL}{\mathbb{L}}

 The first idea of the proof is inspired from \cite{GGPN}, and consists in inverting the operator $\Gamma$ in two steps.
We define the ring $\rL$ formed by elements $f$ of $\mathbb{Q}[p_1,p_2,\dots][[z]]$ such that for all $k\geq 0$, the coefficient of $z^k$ in $f$ is a homogeneous polynomial in the $p_i$ of degree $k$ (where the degree of $p_i$ is defined to be $i$). Equivalently, $\rL=\mathbb{Q}[[zp_1,z^2p_2,z^3 p_3,\dots]]$.
 Note that any formal power series in the Greek variables, considered as an element of $\mathbb{Q}[p_1,p_2,\dots][[z]]$, is an element of $\rL$.
Note also that $L_g$ is an element of $\rL$. Indeed, if we view $L_g$ as a series in $t$, the coefficient of $t^k$ for $k\geq 0$ is a homogenous polynomial of degree $k$ in the $p_i$, since the sum of half-face degrees in a bipartite map is equal to the number of edges. Given the form of the change of variable $t\leftrightarrow z$ given by~\eqref{eq:z}, namely $t=z(1+\sum_k {2k-1 \choose l}p_k z^k)^{-1}$, this clearly implies that as a series in $z$, $L_g$ is in $\rL$.

We now introduce the linear operators $\Box$ and $\Diamond$ on $\mathbb{Q}[x,p_1,p_2,\dots][[z]]$  defined by 
\begin{align*}
\Box x^k = \Big(\frac{1}{k} - \frac{\gamma}{1+\gamma}\Big)p_k,  ~ ~ ~ \Diamond = \sum_{k} p_k \partial_{p_k},
\end{align*} 
where $\partial_{p_k}$ is the differential operator defined by~\eqref{eq:defpartialpk} in \Cref{subsec:defRingsOperators}. 
We have:
\begin{prop}\label{prop:decomposeGamma}
For any $A \in \rL$, we have $$\Diamond A =\Box\Gamma A .$$
In particular, $\Diamond L_g= \Box F_g$.  
\end{prop}
\begin{proof}
The proof is mainly a careful application of the chain rule and of the computations already made in \Cref{sec:Gamma}. Let $R\in\rL$. Since $\Gamma$ is a derivation we have:
\begin{align*}
\Gamma R &= \big(\Gamma z\big) \frac{\partial}{\partial z} R + \sum_{k} \big(\Gamma p_k \big) \partial_{p_k} R
= \big(\Gamma z\big) \frac{\partial}{\partial z} R + \sum_{k\geq 1}  k x^k \partial_{p_k} R, 
\end{align*}
so that $\displaystyle 
 \sum_{k}  k x^k \partial_{p_k} R =\Big(\Gamma - \big(\Gamma z\big) \frac{\partial}{\partial z}\Big) R$.
We now define the linear operators $\Pi: x^k \mapsto p_k$, and $\Xi: x^k \mapsto \frac{p_k}{k}$.
By applying $\Xi$ to the last equality, we get:
\begin{align*}
 \sum_{k}  p_k\partial_{p_k} R 
&= \Xi \Big(\Gamma - \big(\Gamma z\big) \frac{\partial}{\partial z}\Big) R\\
&= \Xi \Gamma  R - \Xi  \big(\Gamma z\big) \frac{\partial}{\partial z}R.
\end{align*}
We thus need to study the expression of $\Xi (\Gamma z) \frac{\partial}{\partial z} R$. We notice that, over the ring $\rL$, the operators $\Pi \sum_{k \geq 1}k \partial_{p_k}$ and $ \frac{z\partial}{\partial z}$ are equal.  Moreover, since $\Gamma z = \frac{(s^{-1} - s) z}{4s^2 (1-\eta)}$ by \Cref{prop:Gamma-on-Greek-expr}, the operator $(\Gamma z)\frac{d}{dz}$ stabilizes $\rL$, so we have:
\[ \Pi \Gamma R = \Pi (\Gamma z) \frac{\partial}{\partial z}R + \Pi \sum_{k \geq 1} kx^k \partial_{p_k}R = \left( \Pi \left( \frac{\Gamma z}{z} \right) + 1\right) z \frac{\partial}{\partial z}R = \frac{1+\gamma}{1-\eta}  \frac{z\partial}{\partial z} R. \] 
(this is the only point in the proof where we use the assumption that $R\in\rL$).
Note that we have used that $\left( \Pi \left( \frac{\Gamma z}{z} \right) + 1\right) = \left(\frac{1}{z}\sum_k kp_k \frac{\partial}{\partial p_k} z \right)+1 = \dfrac{1+\gamma}{1-\eta}$ where the first equality comes from the definition of $\Pi$ and~$\Gamma$, while the second follows from \Cref{prop:diff-of-vars} and the definitions of $\gamma$ and $\eta$.
The last displayed equation thus implies that:
$$  \frac{z\partial}{\partial z} R= \frac{1-\eta}{1+\gamma}\Pi \Gamma R.
$$
Substituting this in the previous expression of $\sum_{k \geq 1} p_k \partial_{p_k R}$ and recalling $\Gamma z = \frac{(s^{-1} - s) z}{4s^2 (1-\eta)}$ we obtain:
\begin{align*} \sum_{k \geq 1} p_k \partial_{p_k}R &=  \Xi \Gamma R -  \Xi\left(\frac{s^{-1} - s}{4 s^2 (1+\gamma)} \right) (\Pi \Gamma)R \\ 
&= \Xi \Gamma  R -  \frac{\gamma}{(1+\gamma)} (\Pi \Gamma)R \\ 
&= \Box\Gamma R,
\end{align*}
where that the last equality is straightforward from the definitions of $\Box, \Pi$, and $\Xi$, while the second one follows from the fact that
 $\Xi \frac{s^{-1}-s}{s^{2}}= D^{-1} \Theta^{-1}\gamma\frac{s^{-1}-s}{s^{2}}\gamma$,
from \Cref{prop:Theta-base} and a direct computation.
 This concludes the proof that $\Diamond R = \Box \Gamma R$ for $R\in\mathbb{L}$.

Finally, since $F_g = \Gamma L_g$ and $L_g \in \mathbb{L}$, it follows that $\Diamond L_g = \Box F_g$.
\end{proof}


\begin{prop}\label{prop:diamondLg}
$\Diamond L_g$ is a rational function of the Greek variables, i.e.:
$\Diamond L_g = R$ with $R\in \mathbb{Q}[\greeks]$, whose denominator is of the form $(1-\eta)^a(1+\zeta)^b$ for $a,b\geq 1$. 
\end{prop}
\begin{proof}
We are going to use \Cref{thm:mainRooted} and the fact that $\Diamond L_g = \Box F_g$. By~\Cref{thm:mainRooted}, and since $F_g$ is $uz$-antisymmetric, we now that $F_g$ is an element of $(s^{-1}-s) \mathbb{Q}(\greeks)[s^2,s^{-2}]$, where as before $s=\frac{1-uz}{1+uz}$. Therefore we can write:
$$
F_g = \sum_{i\in I } (s^{-1} -s )s^{2i} R_i,
$$ 
where  $I\subset \mathbb{Z}$ a finite set of integers and $R_i \in \mathbb{Q}(\greeks)$ is a rational function in the Greek variables for each $i\in I$. 
Since $\Diamond L_g =\Box F_g$  we thus have:
\begin{align}\label{eq:diamondLg}
\Diamond L_g = \sum_{i\in I } R_i  \, \Box \Big((s^{-1} -s )s^{2i} \Big).
\end{align}
Now, by \Cref{prop:Theta-base}, the vector space $(s^{-1}-s)\mathbb{Q}[s^{-2},s^2]$ is spanned by the basis 
$B=
\{\Theta \zeta_i, i\geq 1
 \,;\, 
\Theta (\zeta - \gamma) 
 \,;\, 
\Theta (\eta + \gamma)
 \,;\, 
 \Theta \eta_i, i \geq 1\}$. Moreover, the action of $\Theta$ on the basis $B$ is given by the formulas:
\begin{align}
\Box \Theta \zeta_i\quad\quad 
 &\quad=\quad
X_i - \frac{\gamma \zeta_i}{1+\gamma} \;, 
\label{eq:list1}\\
\Box \Theta (\zeta - \gamma) 
&\quad=\quad
\zeta + \frac{\zeta - \gamma}{1+\gamma}\;, \\
\Box \Theta (\eta + \gamma) 
&\quad= \quad \frac{\gamma (1 - \eta)}{1+\gamma}\;, \\
\Box \Theta \eta_i \quad \quad
&\quad= \quad \eta_{i-1} - \frac{\gamma \eta_i}{1+\gamma}. \label{eq:list4}
\end{align}
where $X_i$ is a linear combination of $\zeta, \zeta_1, \zeta_2 \dots, \zeta_i$ with rational coefficients.
These formulas  follow from the fact that $\Box\Theta : p_k z^k \longmapsto \left(\frac{1}{k}-\frac{\gamma}{1+\gamma}\right) p_k z^k
$
and from the definitions of Greek variables given in the statement of \Cref{thm:mainUnrooted}.
Returning to \eqref{eq:diamondLg}, this proves that $\Diamond L_g$ is a rational function of the Greek variables, $L_g \in \mathbb{Q}[\greeks]$.

\smallskip


Finally, the form of the denominator is clear from the proof and from the form of $F_g$ given by \Cref{thm:mainRooted}.
\end{proof}

\subsection{Inverting $\Diamond$}
\newcommand{\rM}{\mathbb{M}}

Let $S\in \mathbb{Q}(\greeks)$ be a
rational function
in the Greek variables, depending on a finite number of Greek variables. Since each Greek variable is a linear function of the $p_k$, it is clear that $\Diamond$ leaves each Greek variable invariant. Since moreover, $\Diamond$ is a derivation, this implies that $\Diamond S$ is given by a simple \emph{univariate} derivation:
\begin{align}\label{eq:DiamondAsDiff}
\Diamond S = \Big(\frac{d}{dv} S(v \eta, v\gamma, (v\eta_i)_{i\geq 1}, (v\zeta_i)_{i\geq1})\Big)_{v=1}.
\end{align}
This implies:
\begin{prop}\label{LgAsIntegral}
The series $L_g$ is given by:
$$
L_g = \int_0^1 dv R(v \eta, v\gamma, v\zeta, (v\eta_i)_{i\geq 1}, (v\zeta_i)_{i\geq1}).
$$
where $R$ is the rational function such that $\Diamond L_g = \Box F_g = R(\eta, \gamma, \zeta, (\eta_i)_{i\geq 1}, (\zeta_i)_{i\geq 1})$.
\end{prop}
\begin{proof}
We simply integrate \eqref{eq:DiamondAsDiff}. The only thing to check is the initial condition, namely that $R=0$ when all Greek variables are equal to zero. This is clear, since this specialisation is equivalent to substitute $z=0$, and since for $g\geq 1$ there is no map with $0$ edge.
\end{proof}

We thus obtain:
\begin{cor}\label{cor:LgAlmost}
The series $L_g$ has the following form:
$$
L_g = R_1 + R_2 \log (1-\eta) + R_3 \log (1+\zeta)
$$
where $R_1, R_2, R_3$ are rational functions in $(\eta, \gamma, \zeta, (\eta_i)_{i\geq 1}, (\zeta_i)_{i\geq 1})$ depending on finitely many Greek variables. Furthermore, denominator of $R_1$ is of the form $(1-\eta)^a(1+\zeta)^b$ for $a,b\geq 1$. 
\end{cor}
\begin{proof}
This follows from the last two propositions.
\end{proof}


\subsection{Algebraicity and proof of \Cref{thm:mainUnrooted}}
\label{sec:unrooting2}

In order to prove \Cref{thm:mainUnrooted} from \Cref{cor:LgAlmost}, it suffices to show that $R_2=R_3=0$, \textit{i.e.} that no logarithms appear during the integration procedure. In order to do that, it is enough to show that the series $L_g$ is algebraic. We will do this in this section, using a detour via more combinatorial arguments, and using an algebraicity statement proved with bijective methods in~\cite{Chapuy:constellations}. 

The following lemma is a variant for maps of genus $g$ of the ``disymmetry theorem'' classical in the enumeration of labelled trees (and much popularized in the book~\cite{BergeronLabelleLeroux}; see also~\cite{CFKS} for a use in the context of planar maps).
\begin{lemma}[Disymmetry theorem for maps]\label{lemma:disymmetry}
Let $L_g^{vertex}, L_g^{face}, L_g^{edge}$ be the exponential generating function of labelled bipartite maps of genus $g$ with a marked vertex, a marked face, and marked edge, respectively, by the number of edges (variable $t$) and the number of faces of half-degree~$i$ (variable $p_i$, for $i\geq 1$). Then one has:
$$
(2-2g) L_g = L_g^{vertex} +L_g^{face} - L_g^{edge}.
$$
\end{lemma}
\begin{proof}
This is a straightforward consequence of Euler's formula.
\end{proof}

Now we observe that for clear combinatorial reasons, $L_g^{face}$ and $L_g^{edge}$ can be obtained from $F_g$ as follows:
$$
L_g^{face}=\Xi F_g ~ ~ , ~  ~ :L_g^{edge} = \Pi F_g,
$$
where $\Xi: x^k \mapsto \frac{p_k}{k}$ and $\Pi: x^k \mapsto p_k$ are defined as in the previous section. This implies:
\begin{lemma}\label{lemma:faceedge}
$L_g^{face}$ and $L_g^{edge}$ are rational functions of $\eta,\gamma, \zeta, (\eta_i)_{i \geq 1}, (\zeta_i)_{i\geq1}$.
\end{lemma}
\begin{proof}
Given \Cref{thm:mainRooted}, it is enough to prove that $\Xi$ and $\Pi$ send $(s^{-1}-s)\mathbb{Q}[s^{-2},s^2]$ to rational functions of Greek variables.
But by \Cref{prop:Theta-base}, any element $F\in(s^{-1}-s)\mathbb{Q}[s^{-2},s^2]$ is such that $F=\Theta G$ where $G$ is a (finite) linear combination of elements of the basis $B=\{\eta+\gamma;\zeta-\gamma; \eta_i, i\geq 1; \zeta_i, i\geq 1\}$. Now it is clear from the definitions that we have 
$$\Pi\Theta: p_k \mapsto p_k 
  \quad, \quad 
\Xi \Theta: p_k \mapsto \frac{p_k}{k}.
$$
We have to check each of these two operators sends an element of $B$ to a linear combination of Greek variables. For the first one, this is obvious. For the second one, we first observe that from the definition of Greek variables we have from a simple check that $\Xi \Theta (\eta+\gamma)=\gamma$,
$\Xi\Theta(\zeta -\gamma)= 2\zeta-\gamma$, and
$\Xi\Theta(\eta_i)= \eta_{i-1}$ for $i\geq 1$ (with $\eta_0=\eta$). 
Finally, for $i\geq 2$, one similarly checks that there exist rational numbers 
$\alpha_i, \beta_i$ such that $\Xi\Theta \zeta_i = \alpha_i \zeta_i + \beta_i \Xi\Theta \zeta_{i-1}$  which is enough to conclude by induction, together with the base case $\Xi\Theta \zeta_1 = 1/3 (2\zeta_1-2\gamma+4\zeta)$.
\end{proof}

We now need the following result.
\begin{prop}[\cite{Chapuy:constellations}]\label{prop:bijectionImport}
Fix $g\geq 1$ and $D\subset \mathbb{N}$ a finite subset of the integers of maximum at least $2$. Let $\mathbf{p}_D$ denote the substitution $p_i=\mathbf{1}_{i\in D}$ for $i\geq 1$. The the series $L_g^{vertex}(\mathbf{p}_D)$ is algebraic, \textit{i.e} there exists a non-zero polynomial $Q\in\mathbb{Q}[t;p_{i}, i\in D]$ such that
$Q\Big(t; p_i, i \in D; L^{vertex}_g(t; \dots p_i=\mathbf{1}_{i\in D}\dots) \Big)=0$.
\end{prop}
\begin{proof}
Since this statement is not written in this form in \cite{Chapuy:constellations}, let us clarify where it comes from.
Let $O_g\equiv O_g (t;p_1,p_2,\dots)$ be the \emph{ordinary} generating function of \emph{rooted} bipartite maps with one pointed vertex, by the number of edges (variable $t$) and the faces (variable $p_i$ for faces of half-degree $i$, including the root face). Then it is easy to see that we have:
$$
O_g = \frac{t d}{dt} L_g^{vertex}.
$$ 
Now in \cite[Eq. (8.2)]{Chapuy:constellations} (and more precisely in the case $m=2$ of that reference), it is proved that there exists an algebraic series $R_D=R_D(t;p_i,i\in D)$ such that $R_D(t=0)=0$ and
$$
O_g(\mathbf{p}_D)=\frac{t d}{dt} R_D,
$$
where as above $O_g(\mathbf{p}_D)$ is the series $O_g$ under the substitution $p_i=\mathbf{1}_{i \in D}$. Since $L_g^{vertex}(t=0)=0$ for clear combinatorial reasons, we have 
$L_g^{vertex}(\mathbf{p}_D)=R_D$.
\end{proof}

We can now prove \Cref{thm:mainUnrooted}.
\begin{proof}[Proof of \Cref{thm:mainUnrooted}]
For $g\geq 2$, we can conclude from \Cref{lemma:disymmetry}, \Cref{lemma:faceedge} and \Cref{prop:bijectionImport} that for any finite set $D$ of integers with maximum at least $2$ the series $L_g(\mathbf{p}_D)$ is algebraic, where as before $\mathbf{p}_D$ denotes the substitution of variables $p_i=\mathbf{1}_{i\in D}$ for $i\geq 1$.
This implies that the two rational functions $R_2$ and $R_3$ defined in \Cref{cor:LgAlmost}, are vanishing under that specialization:
$$
R_2(\mathbf{p}_D)=0 \quad ; \quad R_3(\mathbf{p}_D)=0.
$$
Therefore to conclude the proof that $R_2=R_3=0$ (hence the proof of \Cref{thm:mainUnrooted}) it is enough to show that if $Q$ is a polynomial in the Greek variables, $Q\in \mathbb{Q}[\greeks]$, such that $Q(p_D)=0$ for all finite $D$ of minimum at least $2$, then $Q=0$.
Take $D=\{L\}$ for $L\geq 2$. Then under $\mathbf{p}_D$, Greek variables are given by:
$$\gamma
\quad ,\quad
\eta = (L-1)\gamma
\quad ,\quad
 \zeta=\frac{L-1}{2L-1} \gamma
\quad ,\quad
 \eta_i = L^i(L-1)\gamma
\quad ,\quad
 \zeta_i=c_i \frac{L}{(2L-1)(2L-3)\dots(2L-2i-1)} \gamma,
$$
for some constants $c_i$.
Therefore if $Q\in \mathbb{Q}[\greeks]$ is a polynomial in Greek variables, we can write $Q$ as $Q=\sum_{i=0}^d Q_i \gamma^i$ for some $d\geq 0$ and polynomials $$Q_i \in \mathbb{Q}\left[L-1,\frac{L-1}{2L-1}; L^i(L-1), 1 \leq i \leq d; \frac{L}{(2L-1)(2L-3)\dots(2L-2i-1)},1 \leq i \leq d \right].$$

Now for fixed $L$, the fact that $Q(\mathbf{p}_{\{L\}})=0$ and that $\gamma = c t +O(t^2)$ (with a constant $c\neq 0$ depending on~$L$)  implies that for each $0\leq i\leq D$ one has $$Q_i\Big(L-1,\frac{L-1}{2L-1}; L^i(L-1), 0 \leq i \leq d; \frac{L}{(2L-1)(2L-3)\dots(2L-2i-1)}, 0 \leq i \leq d \Big)=0.$$
But since this is true for each $L\geq 2$, and since $\{X-1,\frac{X-1}{2X-1}; X^i(X-1), 1 \leq i \leq d; \frac{X}{(2X-1)(2X-3)\dots(2X-2i-1)}, 1 \leq i \leq d\}$ are algebraically independent as rational functions in $X$, this implies that each $Q_i$ is null as a polynomial, so that finally $Q=0$. We thus have proved the rationality of $L_g$ for $g \geq 2$, and by \Cref{cor:LgAlmost}, the denominator of $L_g$ is of the form $(1-\eta)^a (1+\zeta)^b$ for $a,b \geq 1$.

We now prove the bound conditions. Using the three notions of degree in \Cref{subsec:proofMainRooted}, we only need to check that $L_g$ is a homogenous sum of Greek degree $\deg_\gamma(L_g) = 2-2g$ and $\deg_+(L_g) = \deg_-(L_g) \leq 6(g-1)$. We recall the following expression of $F_g$.
\[ F_g = \Gamma L_g = (\Gamma \zeta) \frac{\partial}{\partial \zeta} L_g + (\Gamma \eta) \frac{\partial}{\partial \eta} L_g + \sum_{i \geq 1} (\Gamma \eta_i) \frac{\partial}{\partial \eta_i} L_g + \sum_{i \geq 1} (\Gamma \zeta_i) \frac{\partial}{\partial \zeta_i} L_g \]

For the Greek degree, we observe that, by \Cref{prop:Gamma-degrees} and the fact that $L_g$ has no constant term, if $L_g$ is not homogenous in Greek degree, then $F_g = \Gamma L_g$ cannot be homogenous. Therefore, $L_g$ must be homogenous, with degree $\deg_\gamma(L_g) = \deg_\gamma(F_g) + 1 = 2 - 2g$.

For the pole degree $\deg_+$, let $T = c \eta_\alpha \zeta_\beta (1-\eta)^{-a} (1+\zeta)^{-b}$ for $c \in \rationals$, $a,b \geq 0$ and $\alpha, \beta$ two partitions be the largest term in $L_g$ such that $\deg_+(T) = \deg_+(L_g)$ when ordered first alphabetically by $\alpha$ then also alphabetically by $\beta$. We will now discuss by cases.

If $\alpha$ and $\beta$ are both empty, then $\deg_+(T) = 0$ and we are done. 

We now suppose that $\alpha$ is empty but not $\beta$. We observe that, for a term $S$ in the form $c \zeta_{\beta'} (1-\eta)^{-a} (1+\zeta)^{-b}$, if we order the terms in $\Gamma S$ first by the power of $(1+uz)$ in the denominator then alphabetically by $\nu$ in their factor of the form $\zeta_\nu$, then the largest term $S'$ comes from $(\Gamma \zeta_{\beta'_1}) \partial S / \partial \zeta_{\beta'_1}$, with pole degree $\deg_-(S') = 2|\beta'|+1$ and no possibility of cancellation. Therefore, in $F_g$ there is a term $T'$ coming from $(\Gamma \zeta_{\beta_1}) \partial T / \partial \zeta_{\beta_1}$ that can have no cancellation by the maximality of $\beta$ and by our previous observation, and $\deg_-(T') = 2|\beta|+1$. But since $\deg_-(F_g) \leq 2g-1$, we have $\deg_-(L_g) = 2|\beta| \leq 2g-2$, which concludes this case.

The final case is that $\alpha$ is non-empty. We observe that, for a term $S$ in the form $c \eta_{\alpha'} \zeta_{\beta'} (1-\eta)^{-a} (1+\zeta)^{-b}$, if we order the terms in $\Gamma S$ first by the power of $(1-uz)$ in the denominator then alphabetically by $\nu$ in their factor of the form $\eta_\nu$, then the largest term $S'$ comes from $(\Gamma \eta_{\alpha'_1}) \partial S / \partial \eta_{\alpha'_1}$, with pole degree $\deg_+(S') = 2|\alpha'|+2|\beta'|+5$ and no possibility of cancellation. Therefore, similarly to the previous case, by the fact that $\deg_+(F_g) \leq 6g-1$, we conclude that $\deg_+(L_g) = 2|\alpha|+2|\beta| \leq 6(g-1)$. We thus cover all cases and conclude the proof.
\end{proof}

The only thing that remains now is to address the case of genus $1$:
\begin{proof}[Proof of \Cref{thm:unrootedGenus1}]
We first compute the  series $F_1$, using \Cref{thm:toprec} for $g=1$. Recall that the value of $F_0^{(2)}$ is explicitly given by~\eqref{eq:F02}, and moreover we observe on this expression that $F^{(2)}_{0}$ has a pole of order $4$ at $u=1/z$ and no pole at $u=-1/z$. Therefore, in order to compute the residues in~\eqref{eq:toprec} in the case $g=1$, we need to make explicit, in the expansion of \Cref{prop:diffY},the first $4$ terms at $u=1/z$ and the first $2$ terms at $u=-1/z$. Now, since the proofs of \Cref{prop:taylor-pos-pre} and \Cref{prop:taylor-neg-pre} are computationally effective, we can follow these proofs to compute these quantities explicitly (it is easier, and more reliable, to use a computer algebra system for that). We find:
\begin{multline}
F_1 = \frac{(\eta - 2\eta_1 - 1)/16}{(1-uz)^2(1-\eta)^2}+\frac{4(1+ \zeta)\eta_1 + 3\eta^2 - 6\zeta(1-\eta)  + 3}{96(1-uz) (1+\zeta) (1-\eta)^2}-\frac{1/2}{(1-uz)^5 (1-\eta)}\\
 -\frac{5/4}{(1-uz)^4 (1-\eta)}-\frac{1/32}{(1+uz)(1+\zeta)}-\frac{(21\eta - 2\eta_1 - 21)/24}{(1-uz)^3 (1-\eta)^2}.
\end{multline}
Observe that another approach to prove the last equality is to use the structure given by~\Cref{thm:mainRooted}, and compute sufficiently many terms of the expression of $F_1$ (for example by iterating the Tutte equation~\eqref{eq:Tutte}) to identify all undetermined coefficients appearing in the finite sum~\eqref{eq:mainRooted}.

We now note that all the steps performed to go from \Cref{thm:mainRooted} to \Cref{cor:LgAlmost} are valid when $g=1$, and are computationally effective. Therefore, using the explicit expression of $F_1$ given 
above,
 these steps can be followed, and the expression of $L_1$ obtained explicitly. These computations are automatic (and better performed with a computer algebra system) so we do not print them here.
\end{proof}


\section{Final comments}
\label{sec:comments}

We conclude this paper with several  comments.

First, as explained in the introduction, we have only used two basic ideas from the topological recursion of~\cite{EO}. It may be the case that other features of the latter can be applied to bipartite maps. This may not give stronger structural results than the ones we prove here, but it may provide a different way of performing the "unrooting" step performed in \Cref{sec:unrooting}, similar 
to~\cite[Sec. III-4.2]{Eynard:book}. However the proof we gave has the nice advantage of providing a partly combinatorial explanation on the absence of logarithms in genus $g>1$. More generally, it seems that understanding the link between the disymmetry argument we used here and statements such as~\cite[Thm III~4.2]{Eynard:book} is an interesting question from the viewpoint of the topological recursion itself.

Our next comment is about computational efficiency. While it is tempting to use \Cref{thm:toprec} to compute the explicit expression of $F_g$ (and then $L_g$), it is much easier to simply compute the first few terms of $F_g$ (and $L_g$) using recursively the Tutte equation~\eqref{eq:Tutte}, and then determine the unknown coefficients in ~\eqref{eq:mainUnrooted} or~\eqref{eq:mainRooted} by solving a linear system (recall that~\eqref{eq:mainUnrooted} and~\eqref{eq:mainRooted} are \emph{finite} sums, so there are indeed finitely many coefficients to determine).

 Third, structure results similar to \Cref{thm:mainRooted} for the generating functions $F_g^{(m)}(x_1,x_2,\dots, x_m)$ of bipartite maps of genus $g$ carrying $m\geq 1$ marked faces whose sizes are recorded in the exponents of variables $x_1,x_2,\dots,x_m$, are easily derived from our results. Indeed this series is obtained by applying $m$ times to $L_g$ the rooting operator $\Gamma$, one time in each variable. More precisely:
$$
F_g^{(m)}(x_1,x_2,\dots, x_m)
=\Gamma_1\Gamma_2\dots\Gamma_m L_g,
$$
where $\Gamma_i = \sum_{k\geq 1} k x_i^k \frac{\partial}{\partial p_k}$. Since the action of $\Gamma_i$ is fully described by \Cref{prop:Gamma-on-Greek-expr} (up to replacing $s$ by $s_i=\frac{1-u_iz}{1+u_iz}$, where $u_i = x_i (1+zu_i)^2$), the series $F_g^{(m)}(x_1,x_2,\dots, x_m)$ are easily computable rational functions in the Greek variables and the $(1\pm u_iz)$. 
We observe as well that, by substituting all the $p_i$ to zero in the series 
$F_g^{(m)}(x_1,x_2,\dots, x_m)$, one obtains the generating function of bipartite maps having exactly $m$ faces, where the $x_i$ control the face degrees. Therefore these functions have a nice structure as well, being polynomials in the $1/(1\pm u_iz)$ with rational coefficients.  
This special case also follows from the results of~\cite{KZ} 

 Finally, it is natural to investigate further links between our results and those in \cite{GJV, GGPN}.
One such link is provided by the topological recursion, which is related to all of them, but it seems that even stronger analogies hold between these models.
For example, it is tempting to look for a general model encapsulating all these results. This is the subject of a work in progress.

\bibliographystyle{alpha}
\bibliography{biblio}
\label{sec:biblio}



\end{document}